\documentclass[a4paper,12pt]{article}

\usepackage[utf8]{inputenc}
\usepackage[english]{babel}
\usepackage{enumitem}
\usepackage{amsmath}
\usepackage{amssymb}
\usepackage{enumerate}
\usepackage{amsthm,enumerate}
\usepackage{amsfonts,mathrsfs}
\usepackage{epsfig}
\usepackage{array}
\usepackage{float}
\usepackage{booktabs}
\usepackage{caption}
\usepackage{subcaption}
\usepackage{hyperref}

\usepackage{times}
\usepackage{graphicx}
\usepackage{epstopdf} 
\usepackage[export]{adjustbox} 
\graphicspath{{./Manuscript/}} 
\usepackage{algorithm,algpseudocode}

\usepackage{caption}
\newtheorem{dfn}{Definition}[section]
\newtheorem{lem}[dfn]{Lemma}
\newtheorem{thm}[dfn]{Theorem}
\newtheorem{cor}[dfn]{Corollary}

\theoremstyle{definition}

\newtheorem{asm}[dfn]{Assumption}
\newtheorem{exm}[dfn]{Example}
\newtheorem{rem}[dfn]{Remark}
\newtheorem{prop}[dfn]{Proposition}
\newtheorem{con}[dfn]{Condition}

\newcommand{\R}{\mathbb{R}}

\allowdisplaybreaks

\title{Three-Operator Splitting Method with Two-Step Inertial Extrapolation}

\date{}

\author{Olaniyi S. Iyiola\footnote{Department of Mathematics, Morgan State University, Baltimore, MD, USA; e-mail:niyi4oau@gmail.com; olaniyi.iyiola@morgan.edu},
	\hspace*{0.8mm}
	Lateef O. Jolaoso\footnote{School of Mathematical Sciences, University of Southampton, SO17 1BJ, United Kingdom; Department of Mathematics and Applied Mathematics, Sefako Makgatho Health Sciences University, P.O. Box 94 Medunsa 0204, Pretoria, South Africa; e-mail: l.o.jolaoso@soton.ac.uk.},
\hspace*{0.8mm}
 and
Yekini Shehu,\footnote{College of Mathematics and Computer Science, Zhejiang Normal University, Jinhua,
321004, People's Republic of China; e-mail: yekini.shehu@zjnu.edu.cn.}}

\begin{document}

\maketitle

\begin{abstract}
\noindent The aim of this paper is to study the weak convergence analysis of sequence of iterates generated by a three-operator splitting method of Davis and Yin incorporated with two-step inertial extrapolation for solving monotone inclusion problem involving the sum of two maximal monotone operators and a co-coercive operator in Hilbert spaces. Our results improve on the setbacks observed recently in the literature that one-step inertial Douglas-Rachford splitting method may fail to provide acceleration. Our convergence results also dispense with the summability conditions imposed on inertial parameters and the sequence of iterates assumed in recent results on multi-step inertial methods in the literature. Numerical illustrations from image restoration problem and Smoothly Clipped Absolute Deviation (SCAD) penalty problem are given to show the efficiency and advantage gained by incorporating two-step inertial extrapolation over one-step inertial extrapolation for three-operator splitting method. \\



\end{abstract}

\section{Introduction}
\noindent
Throughout, we take $H$ to be a real Hilbert space where $ \langle . , . \rangle $
and $ \| \cdot \|$ denote scalar product and induced norm respectively. An operator $A:{\rm dom}(A)\subset H\rightarrow 2^H$ is called monotone
if
\begin{equation*}
   \langle \varphi-\psi,x-y\rangle \geq 0~~\forall x,y \in {\rm dom}(A),~~\varphi \in Ax,\psi\in Ay,
\end{equation*}
and if the graph of $A$
$$
   G(A):=\{(x,y):x \in {\rm dom}(A),y \in Ax\}
$$
is not properly contained in the graph of any other monotone operator, we say that $A$ is maximal monotone. Here, domain of $A$ is denoted by ${\rm dom}(A):=\{x\in H|\, Ax\neq\emptyset\}$.
Let $\eta>0$, a mapping $C: H \to H $ is called $\eta-$cocoercive,  if
$$\langle Cu-Cv,u-v\rangle\geq\eta\|Cu-Cv\|^2,\ \forall u,v\in H.$$
\noindent In this paper, we consider an inclusion problem involving three operators namely:
\begin{equation}\label{prob1}
{\rm Find} \ \, x\in H\ \, {\rm such\ that}\ \, 0\in Ax+Bx+Cx,
\end{equation}
where $A,B$ and $C$ are maximal monotone operators defined on $H$ and in addition, $C$ is $\eta-$cocoercive (i.e., $\eta-$inverse strongly monotone). We denote the set of solutions to inclusion problem \eqref{prob1} by zer$(A+B+C)$.
Finding the zero of sum of two operators is an interesting problem to diverse researchers, both in the field of Mathematics and other fields of science.\\

\noindent The forward-backward splitting algorithm (FBS) \cite{Bauschkebook} is an algorithm which has being used to solve \eqref{prob1} when $B\equiv 0$:
\begin{equation}\label{FBM}
x_{n+1}=(1-\lambda_n)x_n+\lambda_nJ_{\gamma A}(I-\gamma C)x_n,
\end{equation}
the operator $J_{\gamma A}=(I+\gamma A)^{-1}$ is the resolvent of $A$ and $\gamma\in (0,2\eta)$.
Algorithm \eqref{FBM} switches between an
explicit step given by operator $C$ with an implicit resolvent step emanating form operator $A$.
Under the condition that the sequence $\{\lambda_n\}\subset [0,\beta]$, where $\beta=2-\frac{\gamma}{2\eta}$; $\displaystyle\sum_{n=1}^{\infty}\lambda_n(\beta-\lambda_n)=+\infty$, it was established \cite[Theorem 26.14]{Bauschkebook} that $\{x_n\}$, the sequence generated by \eqref{FBM}, is weekly convergent to a point $u^*\in {\rm Fix}\big(J_{\gamma A}(I-\gamma C)\big)$ and $u^*$ solves problem \eqref{prob1}, i.e., $u^*\in {\rm zer}(A+C)$. In addition, the sequence $\{x_n\}$  strongly converges to the unique point $u^*\in {\rm zer}(A+C)$ if either $A$ is uniformly monotone on every nonempty bounded subset of dom$(A)$ or $C$ is uniformly monotone on every nonempty bounded subset of $H$. \\

\noindent The inertial method which was introduced by Polyak\cite{Polyakk} in his so-called heavy ball method, has been recently fused into Forward-Backward splitting algorithm, thus helping to speedily increase the rate of  convergence. Moudafi and Oliny \cite{ModOli} added a single valued $\eta$-cocoercive operator $C$  to the inertial proximal point algorithm:

\begin{equation}\label{intFB1}
\begin{cases}
& y_n=x_n+\theta_n(x_n-x_{n-1})\\
& x_{n+1}=J_{\gamma_n A}(y_n-\gamma_nCx_n).
\end{cases}
\end{equation}
It was shown that \eqref{intFB1} converges weakly to solution of \eqref{prob1} (with $B \equiv 0,$) provided $\gamma_n<2\eta$, where $\displaystyle\sum_{n=1}^\infty\theta_n\|x_n-x_{n-1}\|^2<+\infty$. See also \cite{LoPo} for another version of inertial Forward-Backward splitting algorithm. It is shown in \cite[Theorem 2.1]{ModOli} that \eqref{intFB1} converges for $\theta_n \in [0,1)$. In several other inertial forward-backward splitting methods, the convergence analysis is obtained when $\theta_n \in [0,1)$ (see, for example, \cite{LoPo}).\\

\noindent
Other well known algorithms for solving  \eqref{prob1} are the Douglas-Rachford splitting algorithm (DRS) (with $C \equiv 0$) and the Forward-Backward-Forward splitting algorithm (FBFS) (with $C \equiv 0$ and $B$ Lipschitz monotone)\cite{Boyd,ComPesq1,DogRach1,LionsMer1,Passty1,PTsengP}. Some variants of these three splitting algorithm has also been proposed by different researchers (see \cite{Brice,ChamPo,Laurent,VuBC}).\\

\noindent Bot and Csetnek \cite{Bot} introduced the inertial Douglas-Rashford method which is designed to handle inclusion problems in which some of the set-valued mappings involved are composed with linear continuous operators, since in general, there is no closed form of the resolvent of the composition. The inertial Douglas-Rashford method is given as follows:
\begin{equation}
	\begin{cases}
		y_n = J_{\lambda B}(x_n + \theta_n (x_n - x_{n-1}))\\
		z_n = J_{\gamma A}(2y_n - x_n - \theta_n(x_n - x_{n-1}))\\
		x_{n+1} = x_n + \theta_n (x_n - x_{n-1}) + \lambda_n (z_n -y_n), \quad \forall n \geq 1, \label{idr}
	\end{cases}
\end{equation}
and a weak convergence result was proved provided that the sequence $\{\theta_n\}$ and $\{\lambda_n\}$ satisfy the following conditions:
\begin{itemize}
	\item[(i)] $\{\theta_n\}$ is nondecreasing with $\theta_1 =0$ and $0 \leq \theta_n \leq \theta < 1$ $\forall n \geq 1$
	\item[(ii)] $\lambda, \sigma,\delta>0$ such that
	\begin{equation}
		\delta > \frac{\theta^2(1+\theta)+\theta\sigma}{1-\theta^2} \quad \text{and} \quad 0 < \lambda \leq \lambda_n \leq 2\frac{\delta - \theta(\theta(1+\theta)+\theta\delta+\sigma)}{\delta[1+\theta(1+\theta)+\theta\delta+\sigma]} \quad \forall n \geq 1. \nonumber
	\end{equation}
\end{itemize}
Note that when $\theta_n=0$ in \eqref{idr}, the iterative scheme becomes the classical Douglas-Rashford splitting algorithm (see, \cite[Theorem 25.6]{Bauschkebook}):
\begin{equation}
	\begin{cases}
		y_n = J_{\gamma B}x_n \\
		z_n = J_{\gamma A}(2y_n - x_n) \\
		x_{n+1} = x_n + \lambda_n (z_n - y_n),\label{dr}
	\end{cases}
\end{equation}
which converges weakly to a solution of problem \eqref{prob1} (with $C \equiv 0$) provided that the condition $\sum_{n \in \mathbb{N}} \lambda_n (2 - \lambda_n) = +\infty$ is satisfied.  Other forms of inertial Douglas-Rashford method can be found in, for instance, \cite{Alves,Dixit,Fan}.\\

\noindent
\textbf{Advantages of two-step proximal point algorithms.}
In \cite{Poon1,Poon2}, Poon and Liang discussed some limitations of inertial Douglas-Rachford splitting method and inertial ADMM. For example, consider the following feasibility problem in $\mathbb{R}^2$.

\begin{exm}\label{voice}
Let $T_1, T_2 \subset \mathbb{R}^2$ be two subspaces such that $T_1 \cap T_2 \neq \emptyset$. Find $x \in \mathbb{R}^2$ such that $x \in
T_1 \cap T_2$.
\end{exm}
\noindent
It was shown in \cite[Section 4]{Poon2} that two-step inertial Douglas-Rachford splitting method:
$$x_{n+1}=F_{DR}(x_n+\theta(x_n-x_{n-1})+\delta(x_{n-1}-x_{n-2}))$$
converges faster numerically than one-step inertial Douglas-Rachford splitting method
$$x_{n+1}=F_{DR}(x_n+\theta(x_n-x_{n-1}))$$
for Example \ref{voice}, where
$$F_{DR}:=\frac{1}{2}\Big(I+(2P_{T_1}-I)(2P_{T_2}-I)\Big)$$
\noindent is the Douglas-Rachford splitting operator.
In fact, it was shown using this Example \ref{voice} that one-step inertial Douglas-Rachford splitting method
$$x_{n+1}=F_{DR}(x_n+\theta(x_n-x_{n-1}))$$
converges slower than the Douglas-Rachford splitting method $$x_{n+1}=F_{DR}(x_n),$$
This example therefore shows that one-step inertial Douglas-Rachford splitting method may fail to provide acceleration. Therefore, for certain cases, using the inertia of more than two points could be beneficial. It was remark further in \cite[Chapter 4]{Liang1} that the use of more than two points $x_n, x_{n-1}$ could provide acceleration. For example,  the following two-step inertial extrapolation
\begin{equation}\label{afr}
y_n=x_n+\theta(x_n-x_{n-1})+\delta(x_{n-1}-x_{n-2})
\end{equation}
with $\theta > 0$ and $\delta< 0$ can provide acceleration. The failure of one-step inertial acceleration of ADMM was also discussed in \cite[Section 3]{Poon1} and adaptive acceleration for ADMM was proposed instead. Polyak \cite{Polyakbook} also discussed that the multi-step inertial methods can boost the speed of optimization methods even though neither the convergence nor the rate of convergence result of such multi-step inertial methods is established in \cite{Polyakbook}. Some results on multi-step inertial methods have recently been studied in \cite{CombettesGlaudin,DongJOGO}. \\

\noindent
\textbf{Our contribution.}  In this paper, we propose a relaxed Douglas-Rachford splitting method with two-step inertial extrapolation steps to solve monotone inclusion problem \eqref{prob1} and obtain weak convergence results.  The summability conditions of the inertial parameters and the sequence of iterates imposed in \cite[Algorithm 1.2]{CombettesGlaudin}, \cite[Theorem 4.2 (35)]{DongJOGO}, and \cite[Chapter 4, (4.2.5)]{Liang} are dispensed with in our results even for a more general monotone inclusion problem \eqref{prob1}. Some related methods in the literature are recovered as special cases of our proposed method. Finally, we give numerical implementations of our method and compare with other related methods using test problems from image restoration problem and Smoothly Clipped Absolute Deviation (SCAD) penalty problem \cite{SCAD}.\\

\noindent
\textbf{Outline.}
In Section \ref{Sec:Prelims}, we give some basic definitions and results needed in subsequent sections. In Section \ref{Sec3}, we derive our method from the dynamical systems and later introduce our proposed method. We also give both weak convergence results of our method in Section \ref{Sec3}. 
We give some numerical illustrations in Section \ref{Sec5} and concluding remarks are given in Section \ref{Sec6}.

\section{Preliminaries}\label{Sec:Prelims}

\noindent We present few basic properties and Lemmas that we will apply to proof the convergence of our algorithm.
\begin{lem}\rm\label{lem1}
The following identities hold in $H$:
\\
$(i)~~\| v+ w\|^2 = \|v\|^2+2\langle v,w\rangle + \|w\|^2,\ \forall v,w\in H;$\\
$(ii)~~2\langle v-w,v-u\rangle = \|v-w\|^2 + \|v-u\|^2 - \|u-w\|^2,\ \forall u,v,w\in H.$
\end{lem}

Following the ideas  by Davis and Yin \cite{DaY} for solving inclusion problem \eqref{prob1}, we introduce an operator $S$, for some $\gamma>0$, defined by
\begin{equation}\label{pb1}
T:=J_{\gamma A}\circ(2J_{\gamma B}-I-\gamma C\circ J_{\gamma B})+I-J_{\gamma B}.
\end{equation}
The following proposition holds about $T$ defined in \eqref{pb1}.

\begin{prop} \cite[Proposition 2.1]{DaY}\label{prop1}
Suppose $A$ and $B$ are maximal monotone operators and $C$ is a $\eta-$cocoercive mapping. Let $\gamma\in(0,2\eta)$. Then operator $T$ defined in \eqref{pb1} is a $\beta-$averaged mapping with $\beta:=\dfrac{2\eta}{4\eta-\gamma}<1$ and ${\rm zer}(A+B+C)=J_{\gamma B}(F(T))$. Furthermore, if
$\gamma\in(0,2\eta\epsilon)$, where $\epsilon \in (0,1)$, then for all $w,z \in H$,
\begin{eqnarray}\label{sitdown}
\|Tz-Tw\|^2&\leq&\|z-w\|^2-\frac{1-\beta}{\beta}\|(I-T)z-(I-T)w\|^2\nonumber\\
&&- \gamma\Big(2\eta-\frac{\gamma}{\epsilon}\Big)\|C\circ J_{\gamma B}(z)-C\circ J_{\gamma B}(w)\|^2.
\end{eqnarray}
\end{prop}

\begin{dfn}
A mapping $B$ is demiclosed if for every sequence $\{y_n\}$ that converges weakly to $y^*$ with $\{By_n\}$ strongly converging to $x^*$, then $B(y)=x^*$.
\end{dfn}

\begin{lem}(Goebel and Reich \cite{GobRe})\label{lem4}
Let $C\subset H$ be nonempty, closed and convex and $T:C\rightarrow C$  be a nonexpansive mapping. Then $I-T$ is demiclosed at 0.
\end{lem}

\begin{lem} \label{simple}
Let $x,y,z \in H$ and $a,b \in \mathbb{R}$. Then
\begin{eqnarray*}
\|(1+a)x-(a-b)y-bz\|^2&=& (1+a)\|x\|^2-(a-b)\|y\|^2-b\|z\|^2 \\
&&+(1+a)(a-b)\|x-y\|^2+b(1+a)\|x-z\|^2\\
&&-b(a-b)\|y-z\|^2.
\end{eqnarray*}
\end{lem}

\begin{proof}
\begin{eqnarray*}
\|(1+a)x-(a-b)y-bz\|^2&=&\langle (1+a)x-(a-b)y-bz,(1+a)x-(a-b)y-bz \rangle \\
&=& (1+a)^2\|x\|^2-2(1+a)(a-b)\langle x,y\rangle-2b(1+a)\langle x,z\rangle\\
&&+2b(a-b)\langle y,z\rangle+ (a-b)^2\|y\|^2+b^2\|z\|^2\\
&=& (1+a)^2\|x\|^2-(1+a)(a-b)(\|x\|^2+\|y\|^2-\|x-y\|^2) \\
&&-b(1+a)(\|x\|^2+\|z\|^2-\|x-z\|^2)\\
&&+b(a-b)(\|y\|^2+\|z\|^2-\|y-z\|^2)+ (a-b)^2\|y\|^2+b^2\|z\|^2\\
&=&(1+a)\|x\|^2-(a-b)\|y\|^2-b\|z\|^2\\
&&+(1+a)(a-b)\|x-y\|^2+b(1+a)\|x-z\|^2-b(a-b)\|y-z\|^2.
\end{eqnarray*}

\end{proof}

\section{Our Contributions}\label{Sec3}

\subsection{Motivations from Dynamical Systems}

\noindent Consider the following second order dynamical system \cite{Botdiff}
\begin{equation}\label{situ}
 \ddot{x}(t)+\alpha\dot{x}(t)+\beta(x(t)-T(x(t)))=0,~~x(0)=x_0, \dot{x}(0)=v_0,
\end{equation}
where $\alpha,\beta \geq 0$ and $T$ is a nonexpansive mapping. Let $0<\omega_2<\omega_1$ be two weights such that
$\omega_1+\omega_2=1$, $h>0$ is the time step-size, $t_n=nh$ and $x_n=x(t_n)$. Consider an explicit Euler forward discretization with respect to $T$, explicit discretization of $\dot{x}(t)$, and a weighted sum of explicit and implicit discretization of $\ddot{x}(t)$, we have
\begin{eqnarray}\label{situ1}
&&\frac{\omega_1}{h^2}(x_{n+1}-2x_n+x_{n-1})+\frac{\omega_2}{h^2}(x_n-2x_{n-1}+x_{n-2})\nonumber \\
&&+\frac{\alpha}{h}(x_n-x_{n-1})+\beta(y_n-Ty_n)=0,
\end{eqnarray}
where $y_n$ performs "extrapolation" onto the points $x_n, x_{n-1}$ and $x_{n-2}$, which will be chosen later. We
observe that since $I-T$ is Lipschitz continuous, there is some flexibility
in this choice. Therefore, \eqref{situ1} becomes
\begin{eqnarray}\label{situ2}
  &&x_{n+1}-2x_n+x_{n-1}+\frac{\omega_2}{\omega_1}(x_n-2x_{n-1}+x_{n-2})\nonumber \\
  &&+ \frac{\alpha h}{\omega_1}(x_n-x_{n-1})+\frac{\beta h^2}{\omega_1}(y_n-Ty_n)=0.
\end{eqnarray}
This implies that
\begin{eqnarray}\label{situ3}
  &&x_{n+1}=2x_n-x_{n-1}-\frac{\omega_2}{\omega_1}(x_n-2x_{n-1}+x_{n-2}) \nonumber \\
 &&- \frac{\alpha h}{\omega_1}(x_n-x_{n-1})-\frac{\beta h^2}{\omega_1}(y_n-Ty_n)\nonumber \\
&=&x_n+ (x_n-x_{n-1})-\frac{\omega_2}{\omega_1}(x_n-x_{n-1})+\frac{\omega_2}{\omega_1}(x_{n-1}-x_{n-2}) \nonumber \\
&&- \frac{\alpha h}{\omega_1}(x_n-x_{n-1})-\frac{\beta h^2}{\omega_1}(y_n-Ty_n)\nonumber \\
&=& x_n+ \Big(1-\frac{\omega_2}{\omega_1}-\frac{\alpha h}{\omega_1}\Big)(x_n-x_{n-1})+ \frac{\omega_2}{\omega_1}(x_{n-1}-x_{n-2}) \nonumber \\
&&-\frac{\beta h^2}{\omega_1}(y_n-Ty_n).
\end{eqnarray}
Set
$$
\theta:=1-\frac{\omega_2}{\omega_1}-\frac{\alpha h}{\omega_1},~~\delta:=\frac{\omega_2}{\omega_1},~~ \rho:=\frac{\beta h^2}{\omega_1}.
$$
\noindent Then we have from \eqref{situ3} that
\begin{eqnarray}\label{situ4}
x_{n+1}&=&x_n+\theta(x_n-x_{n-1})+\delta(x_{n-1}-x_{n-2})\nonumber \\
  &&-\rho y_n+\rho_nTy_n.
\end{eqnarray}
Choosing $y_n=x_n+\theta (x_n-x_{n-1})+\delta (x_{n-1}-x_{n-2})$. Then \eqref{situ4} becomes
\begin{eqnarray}\label{situ5}
\left\{  \begin{array}{ll}
      & y_n=x_n+\theta (x_n-x_{n-1})+\delta(x_{n-1}-x_{n-2}),\\
      & x_{n+1}=(1-\rho)y_n+\rho Ty_n
      \end{array}
      \right.
\end{eqnarray}
This is two-step inertial Krasnoselskii-Mann iteration. We intend to apply method  \eqref{situ5} to study monotone inclusion problem \eqref{prob1}  in the next subsection of this paper.

\subsection{Convergence Analysis}

\noindent In this subsection, we present our proposed method, which is an application of \eqref{situ5} to monotone inclusion problem \ref{ALGG} and discuss its convergence.

\begin{asm}\label{assum1} First, we suppose that the following conditions hold:
\begin{itemize}
	\item[(A1)] $A$ and $B$ are set-valued maximal monotone operators and $C$ is a $\eta-$cocoercive mapping.
	\item[(A2)] Let $\gamma\in(0,2\eta)$ and define $\beta:=\dfrac{2\eta}{4\eta-\gamma}<1$.
\item[(A3)] Assume ${\rm zer}(A+B+C)\neq\emptyset$.
	\end{itemize}
\end{asm}

\noindent In order to solve \eqref{prob1}, we consider the following two-step inertial Douglas-Rachford splitting method:
\noindent Pick $x_{-1},x_0,x_1 \in H, \theta \in [0,1), \delta \leq 0$ and generate a sequence $\{x_n\}\subset H$ by
\begin{eqnarray}\label{ALGG}
\begin{cases}
y_n = x_n + \theta(x_n - x_{n-1}) + \delta(x_{n-1} - x_{n-2})\\
w_n = J_{\gamma B}(y_n)\\
z_n = J_{\gamma A}(2w_n - y_n - \gamma Cw_n)\\
x_{n+1} = y_n - \rho w_n + \rho z_n,
\end{cases}
\end{eqnarray}
where the following conditions are fulfilled:
\begin{con}\noindent\label{CON}
	\begin{itemize}
		\item[(i)] $0\leq \theta < \min \{\frac{1}{2}, \frac{1-\beta\rho}{1+\beta\rho}\}$;
		\item[(ii)] $ 0 < \rho < \frac{1}{\beta}$;
		\item[(iii)] $\delta\leq 0$ such that 
\begin{eqnarray*}
\max\Big\{-\frac{(1-\beta\rho-\theta-\beta\theta\rho)}{1-\beta\rho},
\frac{\beta\rho\theta(1+\theta)-(1-\beta\rho)(1-\theta)^2}{1+\theta}\Big\}<\delta;
\end{eqnarray*}
and 
\begin{eqnarray*}
\beta\rho\theta(1+\theta)-(1-\beta\rho)(1-\theta)^2
<(2\theta-\beta\rho+2)\delta+(1-2\beta\rho)\delta^2.
\end{eqnarray*}
	\end{itemize}
\end{con}

We give following remarks on our the conditions imposed on the iterative parameters given in Condition \ref{CON} above.

\begin{rem}\label{beforeAppl}
(a) Observe that Condition \ref{CON} (ii) can be considered as an over-relaxation parameter condition for $\rho$ since $\beta<1$.\\[-1mm]

\noindent
(b) Condition \ref{CON} (iii) is true if $\delta=0$. If furthermore, $\rho:=1, \beta:=\frac{1}{2}$ with $C:=0$ and $A:=0$, then our proposed method \eqref{ALGG} and Condition \ref{CON} reduce to \cite[Proposition 2.1]{Alvarez}.\\[-1mm]

\noindent
(c) Suppose $\theta=0=\delta$, then our proposed method \eqref{ALGG} reduces to \cite[Algorithm 1]{DaY}.\hfill $\Diamond$
\end{rem}

\begin{rem}
In our proposed algorithm \eqref{ALGG}, we allow inertial parameter $\delta$ to be non-positive and inertial parameter $\theta$ to be non-negative. We will show that one can benefit from negative choice of the inertial parameter $\delta$ in our numerical implementations given in Section \ref{Sec5}.

\end{rem}

\noindent We start with the following lemma which establishes the boundedness of our generated sequence of iterates from method \eqref{ALGG}.

\begin{lem}\label{lem31}
	The sequence $\{x_n\}$ generated by \eqref{ALGG} is bounded when Assumption \ref{assum1} and Condition \ref{CON} are satisfied.
\end{lem}

\begin{proof}
Following the ideas of Davis and Yin \cite{DaY}, our proposed method \eqref{ALGG} can be converted to a fixed point iteration of the form:
\begin{eqnarray}\label{ALG}
\begin{cases}
	y_n = x_n + \theta(x_n - x_{n-1}) + \delta(x_{n-1} - x_{n-2})\\
	x_{n+1} = (1-\rho)y_n + \rho Ty_n,
\end{cases}
\end{eqnarray}
where
$T:= I - J_{\gamma B} + J_{\gamma A} \circ (2J_{\gamma B} - I - \gamma C \circ J_{\gamma B})$ satisfying \eqref{sitdown}.
Let $x^*\in F(T).$ Then
\begin{eqnarray*}
y_n &=& x_n + \theta(x_n - x_{n-1}) + \delta(x_{n-1} - x_{n-2}) - x^*\\
&=& (1+\theta)(x_n - x^*) - (\theta - \delta)(x_{n-1} - x^*) - \delta(x_{n-2} - x^*).
\end{eqnarray*}
Consequently, we have by Lemma \ref{simple}, that
\begin{eqnarray}\label{YEK}
\|y_n - x^*\|^2 &=& \|(1+\theta)(x_n - x^*) - (\theta - \delta)(x_{n-1} - x^*) - \delta(x_{n-2} - x^*)\|^2\nonumber\\
&=& (1+\theta)\|x_n - x^*\|^2 - (\theta - \delta)\|x_{n-1} - x^*\|^2 - \delta\|x_{n-2} - x^*\|^2\nonumber\\
&&\;\;+ (1+\theta)(\theta - \delta)\|x_n - x_{n-1}\|^2 + \delta(1-\theta)\|x_n - x_{n-2}\|^2\nonumber\\
&&\;\;-\delta(\theta - \delta)\|x_{n-1} - x_{n-2}\|^2.
\end{eqnarray}
Observe that
\begin{eqnarray*}
2\theta \langle x_{n+1}-x_n, x_n-x_{n-1}\rangle &=& 2 \langle \theta(x_{n+1}-x_n), x_n-x_{n-1}\rangle \nonumber \\
&\leq&2|\theta| \|x_{n+1}-x_n\|\|x_n-x_{n-1}\|\nonumber\\
&=&2\theta \|x_{n+1}-x_n\|\|x_n-x_{n-1}\|
\end{eqnarray*}
and so
\begin{equation}\label{happy1}
-2\theta \langle x_{n+1}-x_n, x_n-x_{n-1}\rangle \geq -2\theta \|x_{n+1}-x_n\|\|x_n-x_{n-1}\|.
\end{equation}
Also,
\begin{eqnarray*}
2\delta \langle x_{n+1}-x_n,x_{n-1}-x_{n-2} \rangle &=& 2\langle \delta(x_{n+1}-x_n),x_{n-1}-x_{n-2} \rangle \nonumber \\
&\leq&2|\delta| \|x_{n+1}-x_n\|\|x_{n-1}-x_{n-2}\|
\end{eqnarray*}
which implies that
\begin{equation}\label{happy2}
-2\delta \langle x_{n+1}-x_n,x_{n-1}-x_{n-2} \rangle  \geq -2|\delta| \|x_{n+1}-x_n\|\|x_{n-1}-x_{n-2}\|.
\end{equation}
Similarly, we note that
\begin{eqnarray*}
2\delta\theta \langle x_{n-1}-x_n,x_{n-1}-x_{n-2}\rangle &=& 2\langle \delta\theta(x_{n-1}-x_n),x_{n-1}-x_{n-2}\rangle \nonumber \\
&\leq&2|\delta|\theta \|x_{n-1}-x_n\|\|x_{n-1}-x_{n-2}\|\nonumber\\
&=&2|\delta|\theta \|x_n-x_{n-1}\|\|x_{n-1}-x_{n-2}\|
\end{eqnarray*}
and thus,
\begin{eqnarray}\label{happy3}
2\delta\theta \langle x_n-x_{n-1},x_{n-1}-x_{n-2}\rangle &=& -2 \delta\theta \langle x_{n-1}-x_n,x_{n-1}-x_{n-2}\rangle  \nonumber \\
&\geq& -2|\delta|\theta \|x_n-x_{n-1}\|\|x_{n-1}-x_{n-2}\|.
\end{eqnarray}
By \eqref{happy1}, \eqref{happy2} and \eqref{happy3}, we obtain
\begin{eqnarray}\label{YEK1}
\|x_{n+1}-y_n\|^2 &=& \|x_{n+1}-(x_n+\theta(x_n-x_{n-1})+\delta(x_{n-1}-x_{n-2}))\|^2 \nonumber\\
&=&\|x_{n+1}-x_n-\theta(x_n-x_{n-1})-\delta(x_{n-1}-x_{n-2})\|^2 \nonumber\\
&=&\|x_{n+1}-x_n\|^2-2\theta \langle x_{n+1}-x_n, x_n-x_{n-1}\rangle \nonumber\\
&&-2\delta \langle x_{n+1}-x_n,x_{n-1}-x_{n-2} \rangle+\theta^2\|x_n-x_{n-1}\|^2\nonumber\\
&&+2\delta\theta \langle x_n-x_{n-1},x_{n-1}-x_{n-2}\rangle+\delta^2\|x_{n-1}-x_{n-2}\|^2\nonumber\\
&\geq& \|x_{n+1}-x_n\|^2-2\theta \|x_{n+1}-x_n\|\|x_n-x_{n-1}\|\nonumber\\
&&-2|\delta| \|x_{n+1}-x_n\|\|x_{n-1}-x_{n-2}\|+\theta^2\|x_n-x_{n-1}\|^2\nonumber\\
&&-2|\delta|\theta \|x_n-x_{n-1}\|\|x_{n-1}-x_{n-2}\|+\delta^2\|x_{n-1}-x_{n-2}\|^2\nonumber\\
&\geq& \|x_{n+1}-x_n\|^2-\theta \|x_{n+1}-x_n\|^2-\theta\|x_n-x_{n-1}\|^2\nonumber\\
&&-|\delta| \|x_{n+1}-x_n\|^2-|\delta|\|x_{n-1}-x_{n-2}\|^2+\theta^2\|x_n-x_{n-1}\|^2\nonumber\\
&&-|\delta|\theta \|x_n-x_{n-1}\|^2-|\delta|\theta\|x_{n-1}-x_{n-2}\|^2+\delta^2\|x_{n-1}-x_{n-2}\|^2\nonumber\\
&=&(1-|\delta|-\theta)\|x_{n+1}-x_n\|^2+(\theta^2-\theta-|\delta|\theta)\|x_n-x_{n-1}\|^2\nonumber\\
&&+(\delta^2-|\delta|-|\delta|\theta)\|x_{n-1}-x_{n-2}\|^2.
\end{eqnarray}
Since $T$ is $\beta$-averaged quasi-nonexpansive, we obtain
\begin{eqnarray}\label{YEK2}
	\|x_{n+1} - x^*\|^2 &=& \|(1-\rho)(y_n - x^*) + \rho(Ty_n - x^*)\|^2\nonumber\\
	&=& (1-\rho)\|y_n - x^*\|^2 + \rho\|Ty_n - x^*\|^2 - \rho(1-\rho)\|y_n - Ty_n\|^2\nonumber\\
	&\leq& (1-\rho)\|y_n - x^*\|^2 + \rho\left[\|y_n - x^*\|^2 - \frac{1-\beta}{\beta}\|y_n-Ty_n\|^2\right]\nonumber\\
	&&\;\;-\rho(1-\rho)\|y_n - Ty_n\|^2\nonumber\\
	&=& \|y_n - x^*\|^2 - \frac{(1-\beta)\rho}{\beta}\|y_n - Ty_n\|^2 - \rho(1-\rho)\|y_n - Ty_n\|^2\nonumber\\
	&=& \|y_n - x^*\|^2 - \left[\frac{(1-\beta)\rho}{\beta} + \rho(1-\rho)\right]\|y_n - Ty_n\|^2\nonumber\\
	&=& \|y_n - x^*\|^2 - \left[\frac{(1-\beta)\rho}{\beta} + \rho(1-\rho)\right]\|x_{n+1} - y_n\|^2\nonumber\\
&=& \|y_n - x^*\|^2 - \left[\frac{(1-\beta)}{\beta\rho} + \frac{1-\rho}{\rho}\right]\|x_{n+1} - y_n\|^2.
 \end{eqnarray}
Using \eqref{YEK} and \eqref{YEK1} in \eqref{YEK2}, we get
\begin{eqnarray*}
\|x_{n+1}  - x^*\|^2 &\leq& (1+\theta)\|x_n - x^*\|^2 - (\theta - \delta)\|x_{n-1} - x^*\|^2 - \delta\|x_{n-2} - x^*\|^2\\
&&\;\;+ (1+\theta)(\theta - \delta)\|x_n - x_{n-1}\|^2 + \delta(1+\theta)\|x_n - x_{n-2}\|^2\\
&&\;\;- \delta(\theta - \delta)\|x_{n-1} - x_{n-2}\|^2 -\left[\frac{(1-\beta)}{\beta\rho}
 + \frac{1-\rho}{\rho}\right](1-|\delta| - \theta)\|x_{n+1} - x_n\|^2\\
&&\;\; - \left[\frac{(1-\beta)}{\beta\rho} + \frac{1-\rho}{\rho}\right](\theta^2-\theta-|\delta|\theta)\|x_n-x_{n-1}\|^2\\
 &&\;\; - \left[\frac{(1-\beta)}{\beta\rho} + \frac{1-\rho}{\rho}\right](\delta^2 - |\delta|- |\delta|\theta)\|x_{n-1} - x_{n-2}\|^2\\
&=& (1+\theta)\|x_n - x^*\|^2 - (\theta - \delta)\|x_{n-1} - x^*\|^2 - \delta\|x_{n-2}  - x^*\|^2\\
&&\;\;+ \left[(1+\theta)(\theta - \delta) - \left(\frac{1-\beta}{\beta\rho} + \frac{(1-\rho)}{\rho}\right)(\theta^2 - \theta - |\delta|\theta)\right]\|x_n - x_{n-1}\|^2\nonumber\\
&&\;\;+ \delta(1+\theta)\|x_n - x_{n-2}\|^2 - \left(\frac{1-\beta}{\beta\rho} + \frac{(1-\rho)}{\rho}\right)(1-|\delta| - \theta)\|x_{n+1} - x_n\|^2\\
&&\;\;- \left[\delta(\theta - \delta) + \left(\frac{1-\beta}{\beta\rho} + \frac{(1-\rho)}{\rho}\right)(\delta^2 - |\delta|- |\delta|\theta)\right]\|x_{n-1} - x_{n-2}\|^2\\
&\leq& (1+\theta)\|x_n - x^*\|^2 - (\theta - \delta)\|x_{n-1} - x^*\|^2 - \delta\|x_{n-2} - x^*\|^2\\
&&\;\;+ \left[(1+\theta)(\theta - \delta) - \left(\frac{1-\beta}{\beta\rho} + \frac{(1-\rho)}{\rho}\right)(\theta^2 - \theta - |\delta|\theta)\right]\|x_n - x_{n-1}\|^2\\
&&\;\;
 - \left(\frac{1-\beta}{\beta\rho} - \frac{(1-\rho)}{\rho}\right)(1-|\delta|- \theta)\|x_{n+1} - x_n\|^2\\
&&\;\;-\left[\delta(\theta - \delta) + \left(\frac{1-\beta}{\beta\rho} + \frac{(1-\rho)}{\rho}\right)(\delta^2 - |\delta|- |\delta|\theta)\right]\|x_{n-1} - x_{n-2}\|^2
\end{eqnarray*}
Therefore,
\begin{eqnarray}\label{flavur}
&& \|x_{n+1} - x^*\|^2 - \theta\|x_n - x^*\|^2 - \delta\|x_{n-1} - x^*\|^2\nonumber\\
&&\;+ \left(\frac{1-\beta}{\beta\rho} + \frac{1-\rho}{\rho}\right)(1-|\delta|- \theta)\|x_{n+1} - x_n\|^2\nonumber\\
&&\leq \|x_n - x^*\|^2 - \theta\|x_{n-1} - x^*\|^2 - \delta\|x_{n-2} - x^*\|^2\nonumber\\
&&\;+ \left(\frac{1-\beta}{\beta\rho} + \frac{1-\rho}{\rho}\right)(1-|\delta| - \theta)\|x_n - x_{n-1}\|^2\nonumber\\
&&\;+ \left((\theta -\delta)(1+\theta) - \left(\frac{1-\beta}{\beta\rho} + \frac{1-\rho}{\rho}\right) (\theta^2 - 2\theta - |\delta|\theta - |\delta| + 1)\right)\|x_n - x_{n-1}\|^2\nonumber\\
&&\;-\left[\delta(\theta - \delta) + \left(\frac{1-\beta}{\beta\rho} + \frac{(1-\rho)}{\rho}\right)(\delta^2 - |\delta|- |\delta|\theta)\right]\|x_{n-1} - x_{n-2}\|^2.
\end{eqnarray}
For each $n\geq 1,$ define
\begin{eqnarray}\label{sheyi}
\Gamma_n &:=& \|x_n - x^*\|^2 - \theta\|x_{n-1} - x^*\|^2 - \delta\|x_{n-2} - x^*\|^2\nonumber\\
&&\;\;+ \left(\frac{1-\beta}{\beta\rho} + \frac{(1-\rho)}{\rho}\right)(1-|\delta|- \theta)\|x_n - x_{n-1}\|^2.
\end{eqnarray}
We first show that $\Gamma_n \geq 0, \ \forall n\geq 1.$ Note that
$$ \|x_{n-1} - x^*\|^2 \leq 2\|x_n - x_{n-1}\|^2 + 2\|x_n - x^*\|^2.$$
Hence,
\begin{eqnarray}\label{YEK3}
\Gamma_n &=& \|x_n - x^*\|^2 - \theta\|x_{n-1} - x^*\|^2  - \delta\|x_{n-2} - x^*\|^2\nonumber\\
&&\;\;+ \left(\frac{1-\beta}{\beta\rho} + \frac{(1-\rho)}{\rho}\right)(1-|\delta|- \theta)\|x_n - x_{n-1}\|^2\nonumber\\
&\geq& \|x_n - x^*\|^2 - 2\theta\|x_n - x_{n-1}\|^2 - 2\theta\|x_n - x^*\|^2\nonumber\\
&&\;\;- \delta\|x_{n-2} - x^*\|^2 + \left(\frac{1-\beta}{\beta\rho} + \frac{(1-\rho)}{\rho}\right)(1-|\delta|- \theta)\|x_n - x_{n-1}\|^2\nonumber\\
&=& (1-2\theta)\|x_n - x^*\|^2 + \left[\left(\frac{1-\beta}{\beta\rho} + \frac{(1-\rho)}{\rho}\right)(1-|\delta|-\theta) - 2\theta\right]\|x_n -x_{n-1}\|^2\nonumber\\
&&\;\;- \delta\|x_{n-2} - x^*\|^2.
\end{eqnarray}
By Condition \ref{CON} (i), (ii) and (iii), we obtain
\begin{eqnarray}\label{YEK4}
	|\delta| &<& 1-\theta - \frac{2\theta}{\left(\frac{1-\beta}{\beta\rho} + \frac{1-\rho}{\rho}\right)}\nonumber \\
             &=& \frac{1-\beta\rho-\theta-\beta\theta\rho}{1-\beta\rho}.
\end{eqnarray}
We then obtain from \eqref{YEK3} and \eqref{YEK4} that $\Gamma_n \geq 0, \ \ \forall n\geq 0.$ Consequently, we obtain from \eqref{flavur} that
\begin{eqnarray}\label{YEK2a}
&&\Gamma_{n+1} - \Gamma_n \leq \left((\theta -\delta)(1+\theta) - \left(\frac{1-\beta}{\beta\rho} + \frac{1-\rho}{\rho}\right)(\theta^2 - 2\theta - |\delta|\theta - |\delta| + 1)\right)\|x_n - x_{n-1}\|^2 \nonumber\\
	&&\;\;-\left[\delta(\theta -\delta)+\left(\frac{1-\beta}{\beta\rho} + \frac{1-\rho}{\rho}\right)(\delta^2 - |\delta| - |\delta|\theta)\right]\|x_{n-1} - x_{n-2}\|^2\nonumber\\
&&\;\;-\left((\theta -\delta)(1+\theta) - \left(\frac{1-\beta}{\beta\rho} + \frac{1-\rho}{\rho}\right)(\theta^2 - 2\theta - |\delta|\theta - |\delta| + 1)\right)(\|x_{n-1} - x_{n-2}\|^2\nonumber\\
&&\;\;-\|x_n - x_{n-1}\|^2) + ((\theta -\delta)(1+\theta) - \left(\frac{1-\beta}{\beta\rho} + \frac{1-\rho}{\rho}\right)(\theta^2 - 2\theta - |\delta|\theta - |\delta| + 1)\nonumber\\
&&-\delta(\theta -\delta)-\left(\frac{1-\beta}{\beta\rho} + \frac{1-\rho}{\rho}\right)(\delta^2 - |\delta| - |\delta|\theta))
\|x_{n-1} - x_{n-2}\|^2\nonumber\\
&=&c_1\left(\|x_{n-1} - x_{n-2}\|^2-\|x_n - x_{n-1}\|^2\right)-c_2\|x_{n-1} - x_{n-2}\|^2,
\end{eqnarray}
where
\begin{eqnarray*}
c_1:=-\left((\theta -\delta)(1+\theta) - \left(\frac{1-\beta}{\beta\rho} + \frac{1-\rho}{\rho}\right)(\theta^2 - 2\theta - |\delta|\theta - |\delta| + 1)\right)
\end{eqnarray*}
\noindent
and
\begin{eqnarray*}
&&c_2:=- \Big((\theta -\delta)(1+\theta) - \left(\frac{1-\beta}{\beta\rho} + \frac{1-\rho}{\rho}\right)(\theta^2 - 2\theta - |\delta|\theta - |\delta| + 1)\nonumber\\
&&-\delta(\theta -\delta)-\Big(\frac{1-\beta}{\beta\rho} + \frac{1-\rho}{\rho}\Big)(\delta^2 - |\delta| - |\delta|\theta)\Big).
\end{eqnarray*}
\noindent
Noting that $|\delta|=-\delta$ since $\delta\leq 0$, we then have that
\begin{eqnarray}\label{nige1}
c_1=-\left((\theta -\delta)(1+\theta) - \left(\frac{1-\beta}{\beta\rho} + \frac{1-\rho}{\rho}\right)(\theta^2 - 2\theta - |\delta|\theta - |\delta| + 1)\right)>0
\end{eqnarray}
\noindent 
which is equivalent to
\begin{eqnarray}\label{nige2}
\frac{\theta(1+\theta)-\left(\frac{1-\beta}{\beta\rho} + \frac{1-\rho}{\rho}\right)(1-\theta)^2}{(1+\theta)\left(1+\frac{1-\beta}{\beta\rho} + \frac{1-\rho}{\rho}\right)}<\delta.
\end{eqnarray}
By Condition \ref{CON} (iii), we see that \eqref{nige2} holds and thus $c_1>0$. 
Also,
\begin{eqnarray}\label{nige3}
&&c_2:=- \Big((\theta -\delta)(1+\theta) - \left(\frac{1-\beta}{\beta\rho} + \frac{1-\rho}{\rho}\right)(\theta^2 - 2\theta - |\delta|\theta - |\delta| + 1)\nonumber\\
&&-\delta(\theta -\delta)-\Big(\frac{1-\beta}{\beta\rho} + \frac{1-\rho}{\rho}\Big)(\delta^2 - |\delta| - |\delta|\theta)\Big)>0
\end{eqnarray}
\noindent
implies that
\begin{eqnarray}\label{nige4}
&&\theta(1+\theta)-\left(\frac{1-\beta}{\beta\rho} + \frac{1-\rho}{\rho}\right)(1-\theta)^2
<\left(1+\frac{1-\beta}{\beta\rho} + \frac{1-\rho}{\rho}\right)\delta(1+\theta)\nonumber\\
&&\;\;+\delta(\theta-\delta)+\left(\frac{1-\beta}{\beta\rho} + \frac{1-\rho}{\rho}\right)(\delta^2+\delta(1+\theta)).
\end{eqnarray}
By Condition \ref{CON} (iii), we have that the inequality \eqref{nige4} is satisfied. Therefore, $c_2>0$ from \eqref{nige3}.  
From \eqref{YEK2a}, we then obtain
\begin{eqnarray}\label{ade7}
\Gamma_{n+1}+c_1\|x_n-x_{n-1}\|^2&\leq& \Gamma_{n}+c_1\|x_{n-1}-x_{n-2}\|^2\nonumber \\
&&-c_2\|x_{n-1}-x_{n-2}\|^2.
\end{eqnarray}
Letting $\bar{\Gamma}_n:=\Gamma_n+c_1\|x_{n-1}-x_{n-2}\|^2$, we obtain from \eqref{ade7} that
\begin{equation}\label{afikun}
\bar{\Gamma}_{n+1} \leq \bar{\Gamma}_{n}.
\end{equation}
This implies from \eqref{afikun} that the sequence $\{\bar{\Gamma}_{n}\}$ is decreasing and thus $\underset{n\rightarrow \infty}\lim \bar{\Gamma}_{n}$ exists. Consequently, we have from \eqref{ade7} that
\begin{eqnarray}\label{ade9}
\underset{n\rightarrow \infty}\lim  c_2\|x_{n-1}-x_{n-2}\|^2=0.
	\end{eqnarray}
Hence,
\begin{eqnarray}\label{ade10}
\underset{n\rightarrow \infty}\lim \|x_{n-1}-x_{n-2}\|=0.
	\end{eqnarray}
Using \eqref{ade10} and existence of limit of $\{\bar{\Gamma}_{n}\}$, we have that
\begin{eqnarray}\label{omilomi2}
\underset{n\rightarrow \infty}\lim \Gamma_n&:=& \underset{n\rightarrow \infty}\lim\Big[\|x_n - x^*\|^2 - \theta\|x_{n-1} - x^*\|^2 - \delta\|x_{n-2} - x^*\|^2\nonumber\\
&&\;\;+ \left(\frac{1-\beta}{\beta\rho} + \frac{(1-\rho)}{\rho}\right)(1-|\delta|- \theta)\|x_n - x_{n-1}\|^2\Big]
\end{eqnarray}
exists. Also,
\begin{eqnarray*}
	\|x_{n+1} - y_n\| &=& \|x_{n+1} - x_n - \theta(x_n - x_{n-1}) - \delta(x_{n-1} - x_{n-2})\|\\
	&\leq& \|x_{n+1} - x_n\| + \theta\|x_n - x_{n-1}\| + |\delta|\|x_{n-1} - x_{n-2}\| \to 0, \ n\to \infty.
\end{eqnarray*}
So, we obtain
\begin{eqnarray}\label{disu3}
	\lim_{n\to \infty}\|y_n - Ty_n\|=0.
\end{eqnarray}
Noting that $Ty_n=y_n-w_n+z_n$ from our method \eqref{ALGG}, we have from \eqref{disu3} that
\begin{eqnarray}\label{disu4}
	\lim_{n\to \infty}\|w_n - z_n\|=0.
\end{eqnarray}
Since $x_{n+1} - y_n = \rho(Ty_n - y_n).$ Again, Note that
$$\|y_n - x_n\| \leq \theta\|x_n - x_{n-1}\| + |\delta|\|x_{n-1} - x_{n-2}\| \to 0, \ n\to \infty.$$
Since $\lim_{n\to \infty}\Gamma_n$ exists and $\lim_{n\to \infty}\|x_n - x_{n-1}\| = 0,$ we have from \eqref{YEK3} that $\{x_n\}$ is bounded.
\end{proof}

\begin{thm}\label{t31}
Assume ${\rm zer}(A+B+C)\neq\emptyset$. Let $\gamma\in(0,2\eta\epsilon)$, where $\epsilon \in (0,1)$. Let $\{x_n\}$ be generated by \eqref{ALGG} and Condition \ref{CON} be satisfied. Then the following hold:\\
$(i)~\{x_n\}$ weakly converges to a fixed point $z^*$ of $T$ and $J_{\gamma B}(z^*)$ solves inclusion problem \eqref{prob1};\\
$(ii)~\sum_{n=1}^{\infty}\|Cw_n-Cx^*\|^2< \infty$ for any $x^* \in {\rm zer}(A+B+C)$;\\
$(iii)~\{w_n\}$ and $\{z_n\}$ both converge weakly to $J_{\gamma B}(z^*) \in {\rm zer}(A+B+C)$;\\
$(iv)~\{w_n\}$ and $\{z_n\}$ converge strongly to a point in ${\rm zer}(A+B+C)$ if any of the following is satisfied:
\begin{itemize}
\item[{\rm (a)}] $A$ is uniformly monotone on every nonempty bounded subset of dom($A$) ($A$ is uniformly monotone if there exists
$\phi:\mathbb{R}_+\rightarrow [0,\infty]$ such that $\phi(0)=0$ and for all $u \in Ax, v\in Ay, \langle x-y,u-v\rangle \geq \phi(\|x-y\|)$.);
\item[{\rm (b)}] $B$ is uniformly monotone on every nonempty bounded subset of dom($B$);
\item[{\rm (c)}] $C$ is demiregular at every point $y\in{\rm zer}(A+B+C)$ ($C$ is demiregular at $x\in {\rm dom}(C)$ if for all $u \in Cx$ and for all $(x_n,u_n) \in {\rm gra}(C)$ with $x_n\rightharpoonup x,$ and $u_n\rightarrow u$, we have $x_n\rightarrow x$.
\end{itemize}

\end{thm}

\begin{proof}
(i) Using \eqref{ade10} in \eqref{omilomi2}, we have that
\begin{equation}\label{disu8}
\lim_{n\to \infty} \Big[\|x_n - x^*\|^2 - \theta\|x_{n-1} - x^*\|^2 - \delta\|x_{n-2} - x^*\|^2\Big]
\end{equation}
exists. By Lemma \ref{lem31}, we have that $\{x_n\}$ is bounded. We first show that any weak cluster point of $\{x_n\}$ is in $F(T)$, where $T$ is as defined in \eqref{pb1}. Suppose $\{x_{n_k}\}\subset \{x_n\} $ such that $x_{n_k}\rightharpoonup v^* \in H$. Since $\|y_n - x_n\| \to 0, \ n\to \infty$, we have $y_{n_k}\rightharpoonup v^* \in H$. By the result that $\|y_n - Ty_n\|\to 0, \ n\to \infty$ obtained in Lemma \ref{lem31}, we have that $v^* \in F(T)$.\\

\noindent We now show that $x_n\rightharpoonup x^* \in F(T)$. Let us assume that there exist $\{x_{n_k}\}\subset \{x_n\} $ and $\{x_{n_j}\}\subset \{x_n\} $ such that $x_{n_k}\rightharpoonup v^*, k \rightarrow \infty$ and $x_{n_j}\rightharpoonup x^*, j \rightarrow \infty$. We show that $v^*=x^*$.\\

\noindent Observe that
\begin{equation}\label{away1}
2\langle x_n,x^*-v^*\rangle =\|x_n-v^*\|^2-\|x_n-x^*\|^2-\|v^*\|^2+\|x^*\|^2,
\end{equation}

\begin{eqnarray}\label{away3}
2\langle -\theta x_{n-1},x^*-v^*\rangle &=&-\theta\|x_{n-1}-v^*\|^2+\theta\|x_{n-1}-x^*\|^2\nonumber \\
&&+\theta\|v^*\|^2-\theta\|x^*\|^2
\end{eqnarray}
and
\begin{eqnarray}\label{away3a}
2\langle -\delta x_{n-2},x^*-v^*\rangle &=&-\delta\|x_{n-2}-v^*\|^2+\delta\|x_{n-2}-x^*\|^2\nonumber \\
&&+\delta\|v^*\|^2-\delta\|x^*\|^2.
\end{eqnarray}
Addition of \eqref{away1}, \eqref{away3} and \eqref{away3a} gives
\begin{eqnarray*}
2\langle x_n-\theta x_{n-1}-\delta x_{n-2},x^*-v^*\rangle&=& \Big(\|x_n-v^*\|^2-\theta \|x_{n-1}-v^*\|^2-\delta \|x_{n-2}-v^*\|^2 \Big) \\
  &&-\Big(\|x_n-x^*\|^2-\theta \|x_{n-1}-x^*\|^2-\delta \|x_{n-2}-x^*\|^2 \Big)\\
  &&+(1-\theta-\delta)(\|x^*\|^2-\|v^*\|^2).
\end{eqnarray*}
According to \eqref{omilomi2}, we have
\begin{eqnarray*}
\underset{n\rightarrow \infty}\lim \Big[\|x_n-x^*\|^2 -\theta\|x_{n-1}-x^*\|^2-\delta \|x_{n-2}-x^*\|^2\Big]
\end{eqnarray*}
exists and
\begin{eqnarray*}
\underset{n\rightarrow \infty}\lim \Big[\|x_n-v^*\|^2 -\theta\|x_{n-1}-v^*\|^2-\delta \|x_{n-2}-v^*\|^2 \Big]
\end{eqnarray*}
exists. This implies that
$$
\underset{n\rightarrow \infty}\lim \langle x_n-\theta x_{n-1}-\delta x_{n-2},x^*-v^*\rangle
$$
\noindent exists. Now,
\begin{eqnarray*}
 \langle v^*-\theta v^*-\delta v^*,x^*-v^*\rangle&=& \underset{k\rightarrow \infty}\lim \langle x_{n_k}-\theta x_{n_k-1}-\delta x_{n_k-2},x^*-v^*\rangle  \\
  &=& \underset{n\rightarrow \infty}\lim \langle x_n-\theta x_{n-1}-\delta x_{n-2},x^*-v^*\rangle \\
&=& \underset{j\rightarrow \infty}\lim \langle x_{n_j}-\theta x_{n_j-1}-\delta x_{n_j-2},x^*-v^*\rangle  \\
&=& \langle x^*-\theta x^*-\delta x^*,x^*-v^*\rangle,
\end{eqnarray*}
and this yields
$$
(1-\theta-\delta)\|x^*-v^*\|^2=0.
$$
\noindent Since $\delta \leq 0< 1-\theta$, we obtain that $x^*=v^*$. Hence, $\{x_n\}$ converges weakly to a point in $F(T)$.\\

\noindent
(ii) Observe from \eqref{ALG} that for any $z^* \in F(T)$,
\begin{eqnarray}\label{radio1}
\|x_{n+1} - z^*\|^2 = (1-\rho)\|y_n - z^*\|^2 + \rho\|Ty_n - z^*\|^2 - \rho(1-\rho)\|y_n - Ty_n\|^2
\end{eqnarray}
and by Proposition \ref{prop1} gives
\begin{eqnarray}\label{radio2}
\|Ty_n - z^*\|^2 &\leq& \|y_n - z^*\|^2 - \frac{1-\beta}{\beta}\|y_n-Ty_n\|^2\nonumber \\
&& - \gamma\Big(2\eta-\frac{\gamma}{\epsilon}\Big)\|Cw_n-CJ_{\gamma B}(z^*)\|^2.
\end{eqnarray}
Combining \eqref{radio1} and \eqref{radio2} above, we get
\begin{eqnarray*}
\|x_{n+1} - z^*\|^2 &=& (1-\rho)\|y_n - z^*\|^2 + \rho\|Ty_n - z^*\|^2\nonumber \\
&& - \rho(1-\rho)\|y_n - Ty_n\|^2 \nonumber \\
&\leq& (1-\rho)\|y_n - z^*\|^2+\rho\|y_n - z^*\|^2-\frac{\rho(1-\beta)}{\beta}\|y_n-Ty_n\|^2\nonumber \\
&& - \rho\gamma\Big(2\eta-\frac{\gamma}{\epsilon}\Big)\|Cw_n-CJ_{\gamma B}(z^*)\|^2- \rho(1-\rho)\|y_n - Ty_n\|^2\nonumber \\
&=& \|y_n - z^*\|^2-\frac{\rho(1-\beta)}{\beta}\|y_n-Ty_n\|^2\nonumber \\
&& - \rho\gamma\Big(2\eta-\frac{\gamma}{\epsilon}\Big)\|Cw_n-CJ_{\gamma B}(z^*)\|^2- \rho(1-\rho)\|y_n - Ty_n\|^2.
\end{eqnarray*}
So,
\begin{eqnarray*}
&&\|x_{n+1} - z^*\|^2 +\Big[\frac{\rho(1-\beta)}{\beta}+ \rho(1-\rho)\Big]\|y_n - Ty_n\|^2\nonumber \\
&&+\rho\gamma\Big(2\eta-\frac{\gamma}{\epsilon}\Big)\|Cw_n-CJ_{\gamma B}(z^*)\|^2\nonumber \\
&\leq& \|y_n - z^*\|^2 \nonumber \\
&=& \|x_n + \theta(x_n - x_{n-1}) + \delta(x_{n-1} - x_{n-2}) - z^*\|^2 \nonumber \\
&=& (1+\theta)\|x_n - z^*\|^2 - (\theta - \delta)\|x_{n-1} - z^*\|^2 - \delta\|x_{n-2} - z^*\|^2\nonumber\\
&&+ (1+\theta)(\theta - \delta)\|x_n - x_{n-1}\|^2 + \delta(1-\theta)\|x_n - x_{n-2}\|^2\nonumber\\
&&-\delta(\theta - \delta)\|x_{n-1} - x_{n-2}\|^2.
\end{eqnarray*}
Therefore,

\begin{eqnarray*}
&&\rho\gamma\Big(2\eta-\frac{\gamma}{\epsilon}\Big)\|Cw_n-CJ_{\gamma B}(z^*)\|^2\leq \|x_n - z^*\|^2 \nonumber \\
&&-\|x_{n+1} - z^*\|^2+\theta\Big(\|x_n - z^*\|^2-\|x_{n-1} - z^*\|^2 \Big)+\delta\Big(\|x_{n-1} - z^*\|^2-\|x_{n-2} - z^*\|^2 \Big)\nonumber \\
&&+ (1+\theta)(\theta - \delta)\|x_n - x_{n-1}\|^2 + \delta(1-\theta)\|x_n - x_{n-2}\|^2\nonumber\\
&&-\delta(\theta - \delta)\|x_{n-1} - x_{n-2}\|^2\nonumber\\
&\leq& \|x_n - z^*\|^2 -\|x_{n+1} - z^*\|^2+\theta\Big(\|x_n - z^*\|^2-\|x_{n-1} - z^*\|^2 \Big)\nonumber\\
&&-\delta\Big(\|x_{n-2} - z^*\|^2-\|x_{n-1} - z^*\|^2\Big)+ (1+\theta)(\theta - \delta)\|x_n - x_{n-1}\|^2\nonumber\\
&&-\delta(\theta - \delta)\|x_{n-1} - x_{n-2}\|^2.
\end{eqnarray*}
Summing, we obtain
\begin{eqnarray}\label{disu9}
&&\rho\gamma\Big(2\eta-\frac{\gamma}{\epsilon}\Big)\sum_{k=1}^{n}  \|Cw_k-CJ_{\gamma B}(z^*)\|^2\leq \|x_1 - z^*\|^2 \nonumber \\
&&-\theta\|x_0 - z^*\|^2-\delta\|x_{-1} - z^*\|^2-\|x_{n+1} - z^*\|^2\nonumber \\
&&+\theta\|x_n - z^*\|^2+\delta\|x_{n-1} - z^*\|^2+(1+\theta)(\theta - \delta)\sum_{k=1}^{n} \|x_k - x_{k-1}\|^2\nonumber\\
&&-\delta(\theta - \delta)\sum_{k=1}^{n}\|x_{k-1} - x_{k-2}\|^2\nonumber\\
&=&\|x_1 - z^*\|^2 -\theta\|x_0 - z^*\|^2-\delta\|x_{-1} - z^*\|^2\nonumber\\
&&-\Big(\|x_{n+1} - z^*\|^2-\theta\|x_n - z^*\|^2-\delta\|x_{n-1} - z^*\|^2 \Big)\nonumber\\
&&+(1+\theta)(\theta - \delta)\sum_{k=1}^{n} \|x_k - x_{k-1}\|^2\nonumber\\
&&-\delta(\theta - \delta)\sum_{k=1}^{n}\|x_{k-1} - x_{k-2}\|^2.
\end{eqnarray}
Using \eqref{ade10} and \eqref{disu8} in \eqref{disu9}, we obtain
$$
\sum_{n=1}^{\infty}\|Cw_n-CJ_{\gamma B}(z^*)\|^2< \infty.
$$
\noindent This completes (ii).\\

\noindent
(iii)
Since $J_{\gamma B}$ is nonexpansive, we have
\begin{eqnarray*}
\|w_n-J_{\gamma B}(z^*)\| &=& \|J_{\gamma B}(y_n)-J_{\gamma B}(z^*)\| \\
  &\leq& \|y_n-z^*\| = \|(1+\theta)(x_n - z^*) - (\theta - \delta)(x_{n-1} - z^*) - \delta(x_{n-2} - z^*)\|\\
  &\leq& (1+\theta)\|x_n - z^*\|+|\theta - \delta| \|x_{n-1} - z^*\|+|\delta| \|x_{n-2} - z^*\|.
\end{eqnarray*}
By the boundedness of $\{x_n\}$ in Lemma \ref{lem31}, we then have that $\{w_n\}$ is bounded . Now, suppose $z$ is the sequential weak cluster point of  $\{w_n\}$. Then there exists $\{w_{n_k}\} \subset \{w_n\}$ such that $w_{n_k} \rightharpoonup z, k\rightarrow \infty$. Pick $x^* \in {\rm zer}(A+B+C)$ and since $C$ is maximal monotone, $Cw_n\rightarrow Cx^*$ and $w_{n_k} \rightharpoonup z$, it follows from the weak-to-strong sequential closedness of $C$ that $Cz=Cx^*$ by \cite[Proposition 20.33 (ii)]{Bauschkebook} and so $Cw_{n_k} \rightarrow Cz$. \\

\noindent Denote $u_n:=\frac{1}{\gamma}(y_n-w_n) \in Bw_n$, and $v_n:=\frac{1}{\gamma}(2w_n-y_n-\gamma Cw_n-z_n) \in Az_n$. \\

\noindent Then we get
$w_{n_k} \rightharpoonup z$, $z_{n_k} \rightharpoonup z$, $Cw_{n_k} \rightarrow Cz$, $y_{n_k} \rightharpoonup z^*$, $u_{n_k} \rightharpoonup \frac{1}{\gamma}(z^*-z)$ and $v_{n_k} \rightharpoonup \frac{1}{\gamma}(z-z^*-\gamma Cz)$. Applying \cite[Proposition 25.5]{Bauschkebook} to
$(z_{n_k},v_{n_k}) \in {\rm gra} A$, $(y_{n_k},u_{n_k}) \in B$ and $(w_{n_k},Cw_{n_k}) \in C$ shows that $z \in  {\rm zer}(A+B+C)$, $z^*-z \in \gamma Bz$ and $z-z^*-\gamma Cz \in \gamma Az$. Hence, we have shown that $z=J_{\gamma B}(z^*)$ and so $z$ is unique weak cluster point of $\{w_n\}$. Therefore, $\{w_n\}$ weakly converges to $J_{\gamma B}(z^*)$ by \cite[Lemma 2.38]{Bauschkebook}.  Furthermore,  $\{z_n\}$ weakly converges to $J_{\gamma B}(z^*)$ since $\{w_n\}$ weakly converges to $J_{\gamma B}(z^*)$ and $w_n-z_n\rightarrow 0, n\rightarrow \infty$.\\

\noindent
(iv) Let $x^*=J_{\gamma B}(z^*), u^*=\frac{1}{\gamma}(z^*-x^*) \in Bx^*$ and $v^*=\frac{1}{\gamma}(x^*-z^*)-Cx^* \in Ax^*$. \\

\noindent (a) Since $B+C$ is monotone and $(y_n,u_n) \in {\rm gra} B$, we get
$$
\langle w_n-x^*,u_n+Cw_n-(u^*+Cx^*) \rangle \geq 0, \forall n \geq 1.
$$
 \noindent Consider the set $Z:=\{x^*\}\cup \{z_n\}$. Then there exists $\Phi_A:\mathbb{R}_+\rightarrow [0,\infty]$ that vanishes only at 0 such that
\begin{eqnarray*}
\gamma \Phi_A(\|z_n-x^*\|)&\leq & \gamma \langle z_n-x^*,v_n-v^* \rangle\nonumber \\
&&+\gamma \langle w_n-x^*,u_n+Cw_n-(u^*+Cx^*) \rangle\nonumber \\
&=&\gamma \langle z_n-w_n,v_n-v^*\rangle+ \gamma \langle z_n-x^*,v_n-v^*\rangle\nonumber \\
&&+\gamma \langle w_n-x^*,u_n+Cw_n-(u^*+Cx^*) \rangle\nonumber \\
&=&\gamma \langle z_n-w_n,v_n-v^*\rangle+ \gamma \langle w_n-x^*,v_n+u_n+Cw_n\rangle\nonumber \\
&=& \langle w_n-z_n,w_n-\gamma v_n-(x^*-\gamma v^*)\rangle\nonumber \\
&=& \langle w_n-z_n,y_n-z^*\rangle + \gamma \langle w_n-z_n,Cw_n-Cx^*\rangle\nonumber \\
&&- \|w_n-z_n\|^2 \nonumber \\
&\leq& \langle w_n-z_n,y_n-z^*\rangle + \gamma \langle w_n-z_n,Cw_n-Cx^*\rangle\rightarrow 0, n\rightarrow \infty,
\end{eqnarray*}
 since $w_n-z_n=y_n-Ty_n\rightarrow 0, n\rightarrow \infty, y_n\rightharpoonup z^*$ and $Cw_n\rightarrow Cx^*, n\rightarrow \infty$. Also,
 $w_n\rightarrow x^*$ since $z_n-w_n\rightarrow 0, n\rightarrow \infty$. \\

\noindent (b) Since $A$ is monotone, we get $\langle z_n-x^*,v_n-v^* \rangle \geq 0, n \geq 1$. Observe also that $B+C$ is uniformly monotone on all bounded sets. Let us consider the bounded set $Z:=\{x^*\} \cup \{w_n \}$. Then there exists an increasing function $\Phi_B:\mathbb{R}_+\rightarrow [0,\infty]$ that vanishes only at 0 such that
\begin{eqnarray*}
 \gamma\Phi_B(\|w_n-x^*\|)&\leq & \gamma \langle z_n-x^*,v_n-v^* \rangle\nonumber \\
&&+\gamma \langle w_n-x^*,u_n+Cw_n-(u^*+Cx^*) \rangle\rightarrow 0, n\rightarrow \infty,
\end{eqnarray*}
 by part (a) above. Hence we get $w_n\rightarrow x^*, n\rightarrow \infty$.\\

\noindent (c) Observe that $Cw_n\rightarrow Cx^*$ and $w_n \rightharpoonup x^*$. Then by demiregularity of $C$, we get  $w_n\rightarrow x^*, n\rightarrow \infty$.

\end{proof}

\noindent We obtain the following deduced corollaries from our main results in Theorem \ref{t31}.

\begin{cor}
Suppose $A:H\rightarrow 2^H$ is a set-valued maximal monotone operator and $C:H\rightarrow H$ is a $\eta-$cocoercive mapping. Assume that ${\rm zer}(A+C)\neq\emptyset$. Let $\gamma\in(0,2\eta\epsilon)$, where $\epsilon \in (0,1)$  and choose $x_{-1},x_0,x_1 \in H$ with Condition \ref{CON} satisfied and generate $\{x_n\}\subset H$ by
\begin{eqnarray}\label{load1}
\begin{cases}
w_n = x_n + \theta(x_n - x_{n-1}) + \delta(x_{n-1} - x_{n-2})\\
x_{n+1} = (1- \rho) w_n + \rho J_{\gamma A}(w_n -\gamma Cw_n).
\end{cases}
\end{eqnarray}
Then both $\{w_n\}$ and $\{x_n\}$ converge weakly to $z^* \in {\rm zer}(A+C)$ under Condition \ref{CON}. Furthermore, $\displaystyle\sum_{n=1}^{\infty}\|Cw_n-Cz^*\|^2< \infty$ for any $z^* \in {\rm zer}(A+C)$.
\end{cor}

\section{Numerical Implementations}\label{Sec5}
\noindent This section is focused on numerical implementations of our proposed algorithms. We apply our method \eqref{ALGG} to solve the Least Absolute Selection and Shrinkage Operator (LASSO) problem \cite{LASSO} in compressed sensing, Smoothly Clipped Absolute Deviation (SCAD) penalty problem \cite{SCAD} and image restoration problem. Furthermore, we compare the performance of our scheme with some other related schemes in \cite{Aragon-Artacho,Beck,CevVu,DaY,enyi,ModOli} which have been proposed previously in the literature.

\begin{exm}
\subsubsection*{Image Restoration Problem}\label{Sec:LASSO}
Let us consider the following optimization problem (\cite{Beck}):

\begin{equation}\label{linprob}
\displaystyle\min_{u\in \R^N}\{f(u)+g(u)\}
\end{equation}
under the following assumptions for the functions $f$ and $g$ (\cite{Beck,YekNiFe}):

\begin{asm}
$\,$
\begin{itemize}
\item[{\rm (i)}] $f:\R^N\to \R$ is a convex and continuous function that is not necessarily smooth.
\item[{\rm (ii)}] $g:\R^N \to \R$  is Fr\'{e}chet differentiable and the gradient $\nabla g$ is  L-Lipschitz continuous..
\end{itemize}
\end{asm}
\noindent
Setting $A:=\partial f, \ B := 0$ and $C:=\nabla g$, we have that problem \eqref{linprob} is a special case of \eqref{prob1} and thus we can apply
\eqref{load1}, which is a special case of our proposed method \eqref{ALGG}.\\

\noindent Let $D\in\R^{M\times N}$, $b\in\R^M$ and $\gamma>0$. We state the $l_1$-norm regularization least squares model as

\begin{equation}\label{lsqm}
\displaystyle\min_{u\in\R^N}\left\{\frac{1}{2}\|Du-b\|_2^2 +\gamma\|u\|_1 \right\},
\end{equation}
where $\gamma>0$, $x\in \mathbb{R}^N$ is the original image to be recovered, $b\in \mathbb{R}^M$ is the observed image and $D: \mathbb{R}^N\to \mathbb{R}^M$ is the blurring operator. In order to solve the above problem  via numerical computations, we consider the $256 \times 256$ Cameraman Image, the $128 \times 128$ Medical Resonance Imaging (MRI) and $128 \times 128$ Pout image, all obtained from the MATLAB Image Processing Toolbox. Moreover, we use the Gaussian blur of size $9\times 9$ and standard deviation $\sigma=4$ to create the blurred and noisy image (observed image). Also, we measure the quality of the restored image using the signal-to-noise ratio defined by
\begin{equation}
	\mbox{SNR} = 20 \times \log_{10}\left(\frac{\|x\|_2}{\|x-x^*\|_2}\right), \nonumber
\end{equation}
where $x$ is the original image and $x^*$ is the restored image.
Note that the larger the SNR, the better the quality of the restored image. We also choose the initial values as $x_{-1} = 0.01 \times\textbf{1} \in \mathbb{R}^{N\times N}, x_0 = \textbf{0} \in \mathbb{R}^{N\times N}$ and $x_1 = \textbf{1} \in \mathbb{R}^{N\times M}.$ We compare the performance of proposed Algorithm \eqref{ALGG} with the methods proposed by Artacho \& Belen Alg. in \cite{Aragon-Artacho}, Beck \& Teboulle Alg. in \cite{Beck} and Damek \& Yin Alg. in \cite{DaY} after 1,000 iterations. The parameters for each method for comparison are presented in Table \ref{tab1} below.

\begin{table}[H]
\caption{Methods Parameters for Example \ref{Sec:LASSO}}
\centering
\begin{tabular}{c c c c c c c c c c c c c c}
 \toprule[1.5pt] \\
Proposed Alg. \ref{ALGG} & $\displaystyle \theta = 0.49$ & $\displaystyle \delta = -0.01$  & $\displaystyle \eta = \frac{1}{||D^*D||}$ & $\displaystyle \gamma = \eta$\\
\\
 & $\displaystyle \beta = \frac{2\gamma}{4\eta - \gamma}$ & $\displaystyle \rho = \frac{0.7}{\beta}$\\
\toprule[1.5pt]\\
Artacho \& Belen Alg. & $\displaystyle \eta = \frac{1}{||D^*D||}$ & $\displaystyle \gamma = 3.9\eta$ & $\displaystyle \beta = 2 - \frac{\gamma}{2\eta}$ & $\displaystyle \lambda = 0.7\beta$\\
\toprule[1.5pt] \\
Beck \& Teboulle Alg. & $\displaystyle t_n = \frac{1}{n + 1}$ & $\displaystyle \eta = \frac{1}{||D^*D||}$ & $\displaystyle \gamma = \eta$ \\
\toprule[1.5pt] \\
Damek \& Yin Alg. & $\displaystyle \eta = \frac{1}{||D^*D||}$ & $\displaystyle \gamma = \eta$ & $\displaystyle \beta = \frac{2\gamma}{4\eta - \gamma}$ & $\displaystyle \lambda = \frac{0.7}{\beta}$\\
\toprule[1.5pt]
\end{tabular}
\label{tab1}
\end{table}

\begin{table}[H]
\caption{Comparison of our proposed Algorithm \ref{ALGG} with other algorithms for the Image Restoration problem.}
\centering 
\begin{tabular}{c c c c c c c c c c c} 
\toprule[1.5pt]
 & \multicolumn{2}{c}{Cameraman} & &  \multicolumn{2}{c}{MRI} & &  \multicolumn{2}{c}{Pout} \\
  \cline{2-3}   \cline{5-6}  \cline{8-9}
 & SNR  & CPU Time & & SNR  & CPU Time & & SNR  & CPU Time \\
 \toprule[1.5pt] \\
Proposed Alg. \ref{ALGG} & 35.1935 & 8.8241 && 26.2299 & 1.9834 && 42.3778 & 9.3248  \\ [0.5ex]
 \hline \\
Damek \& Yin Alg. & 34.2577 & 8.8161 && 25.8366 & 1.9524 && 41.1374 & 9.3610  \\ [0.5ex]
 \hline\\
Artacho \& Belen Alg. & 29.1105 & 8.7788 && 22.4081 & 1.9272 && 34.2528 & 9.2409  \\ [0.5ex]
\hline\\
Beck \& Teboulle Alg. & 33.6452 & 8.7521 && 25.5244 & 1.9347 && 40.3326 & 10.1433 \\ [0.5ex]
 \hline
\end{tabular}
\label{table:IRP}
\end{table}

\begin{figure}[H]
    \centering
    \includegraphics[width = 0.9\textwidth]{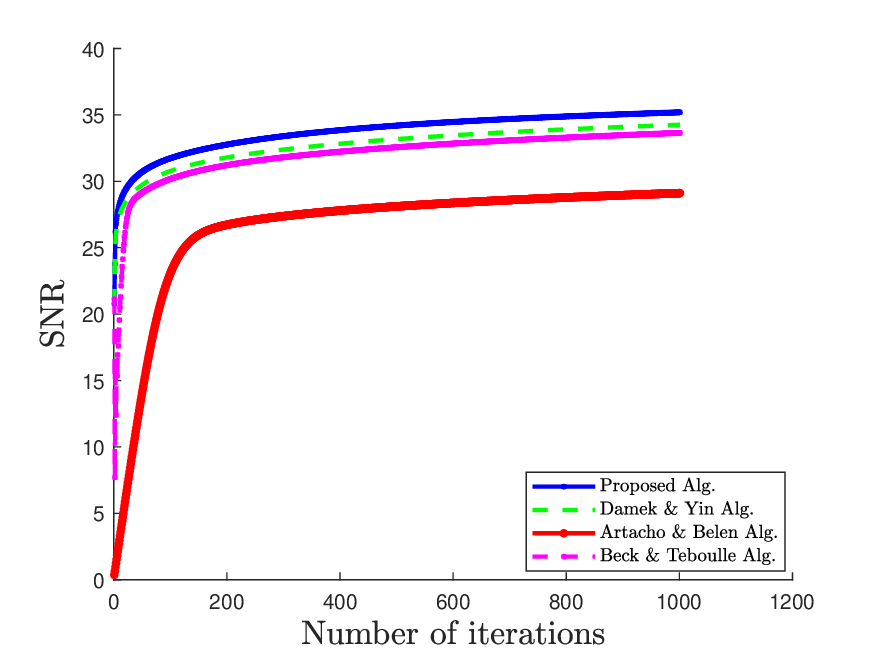}
    \caption{Signal-to-noise ratio (SNR) of different methods for Cameraman image}
    \label{SNR_cam}
\end{figure}

\begin{figure}[H]
    \centering
    \includegraphics[width = 0.9\textwidth]{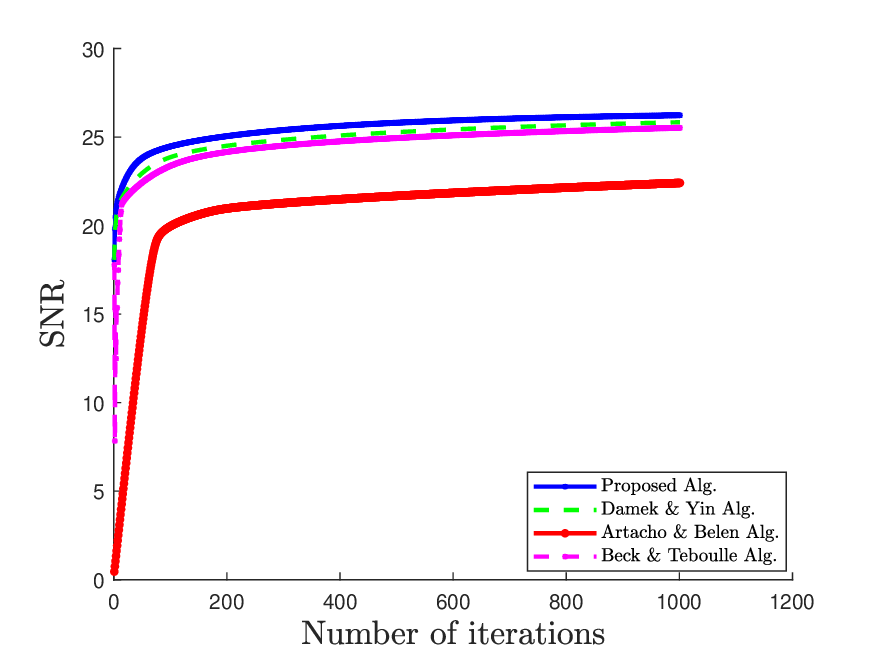}
    \caption{Signal-to-noise ratio (SNR) of different methods for Medical Resonance
Imaging (MRI)}
    \label{SNR_MRI}
\end{figure}

\begin{figure}[H]
    \centering
    \includegraphics[width = 0.9\textwidth]{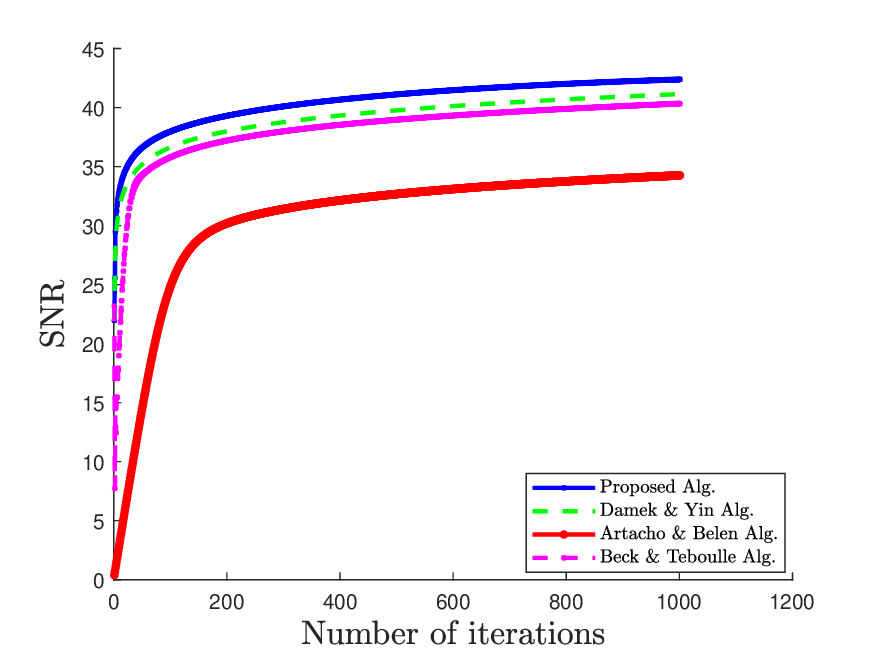}
    \caption{Signal-to-noise ratio (SNR) of different methods for Pout image}
    \label{SNR_Pout}
\end{figure}

\begin{figure}[H]
     \begin{subfigure}[b]{0.57\textwidth}
         \includegraphics[width=\textwidth]{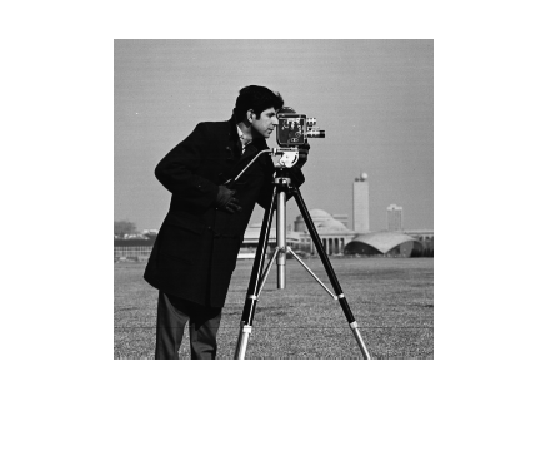}
         \caption{Original Cameraman Image}
     \end{subfigure}
     \begin{subfigure}[b]{0.57\textwidth}
         \includegraphics[width=\textwidth]{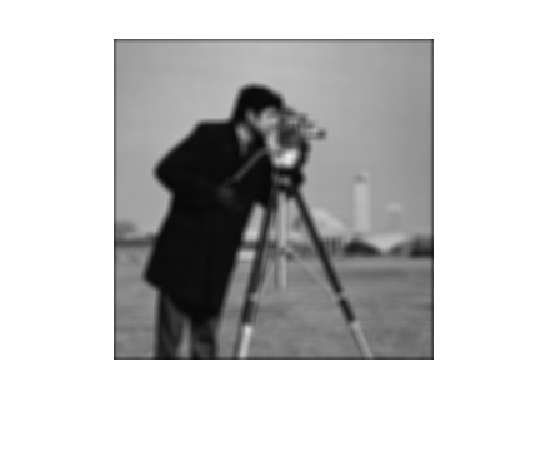}
         \caption{Degraded Cameraman Image}
     \end{subfigure}
     \begin{subfigure}[b]{0.57\textwidth}
         \includegraphics[width=\textwidth]{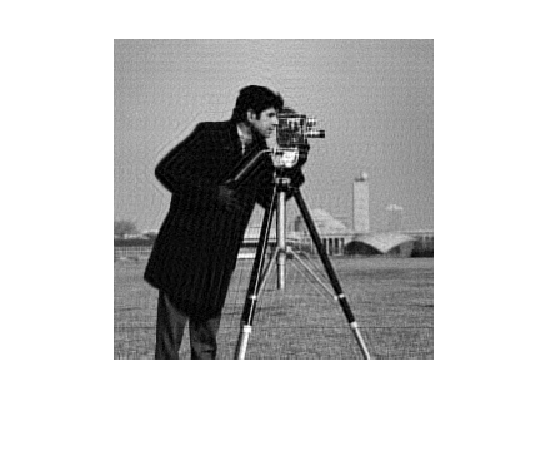}
         \caption{Proposed Alg. \ref{ALGG}}
     \end{subfigure}
     \begin{subfigure}[b]{0.57\textwidth}
         \includegraphics[width=\textwidth]{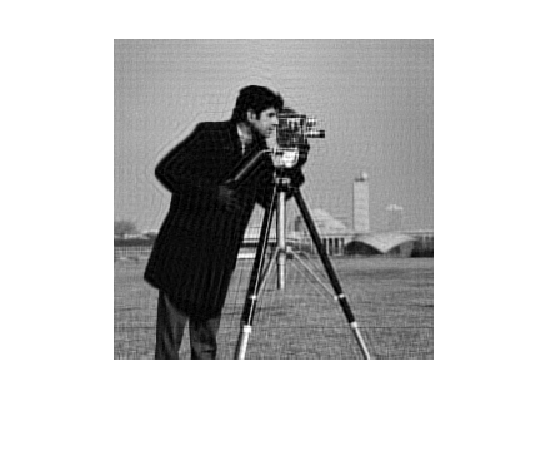}
         \caption{Damek \& Yin Alg.}
     \end{subfigure}
     \begin{subfigure}[b]{0.57\textwidth}
         \includegraphics[width=\textwidth]{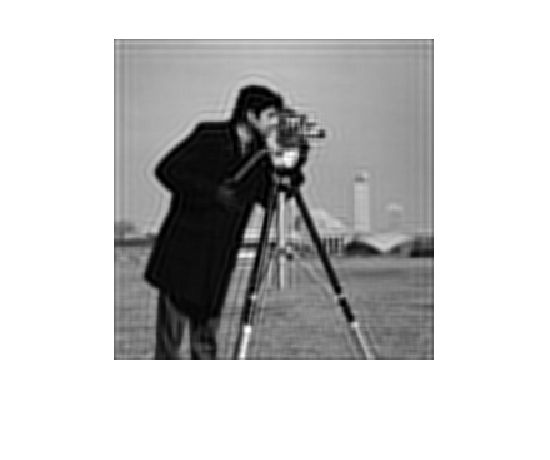}
         \caption{Artacho \& Belen Alg.}
     \end{subfigure}
     \begin{subfigure}[b]{0.57\textwidth}
         \includegraphics[width=\textwidth]{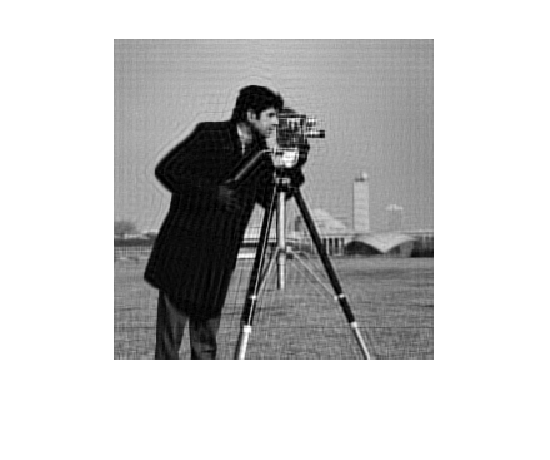}
         \caption{Beck \& Teboulle Alg.}
     \end{subfigure}
        \caption{Example \ref{Sec:LASSO} - Image recovery of Cameraman image by different methods.}
        \label{Cam_image}
\end{figure}

\begin{figure}[H]
     \begin{subfigure}[b]{0.57\textwidth}
         \includegraphics[width=\textwidth]{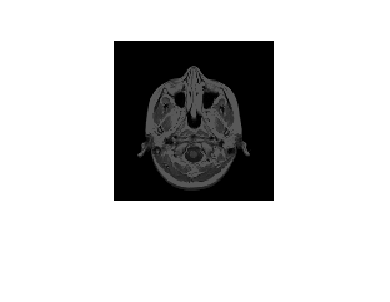}
         \caption{Original MRI}
     \end{subfigure}
     \begin{subfigure}[b]{0.57\textwidth}
         \includegraphics[width=\textwidth]{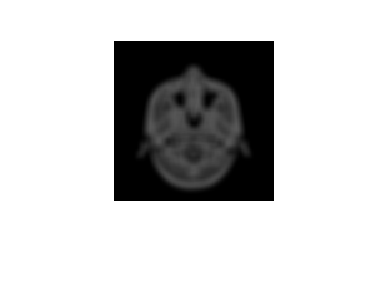}
         \caption{Degraded MRI}
     \end{subfigure}
     \begin{subfigure}[b]{0.57\textwidth}
         \includegraphics[width=\textwidth]{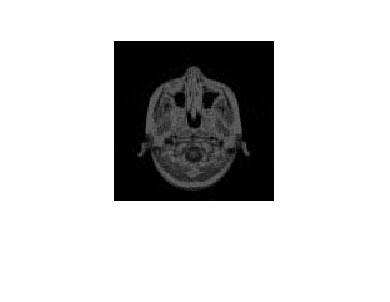}
         \caption{Proposed Alg. \ref{ALGG}}
     \end{subfigure}
     \begin{subfigure}[b]{0.57\textwidth}
         \includegraphics[width=\textwidth]{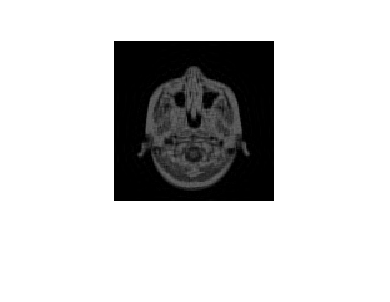}
         \caption{Damek \& Yin Alg.}
     \end{subfigure}
     \begin{subfigure}[b]{0.57\textwidth}
         \includegraphics[width=\textwidth]{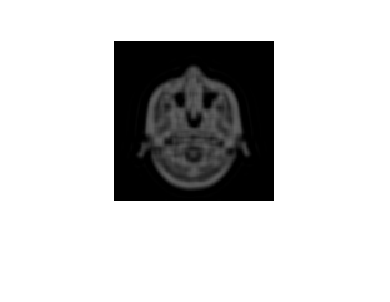}
         \caption{Artacho \& Belen Alg.}
     \end{subfigure}
     \begin{subfigure}[b]{0.57\textwidth}
         \includegraphics[width=\textwidth]{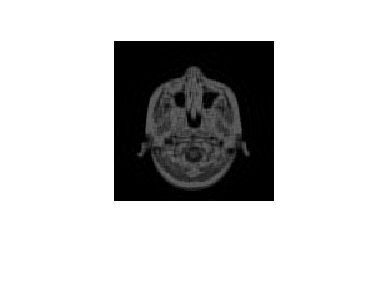}
         \caption{Beck \& Teboulle Alg.}
     \end{subfigure}
        \caption{Example \ref{Sec:LASSO} - Image recovery of MRI by different methods.}
        \label{MRI_image}
\end{figure}

\begin{figure}[H]
     \begin{subfigure}[b]{0.55\textwidth}
         \includegraphics[width=\textwidth]{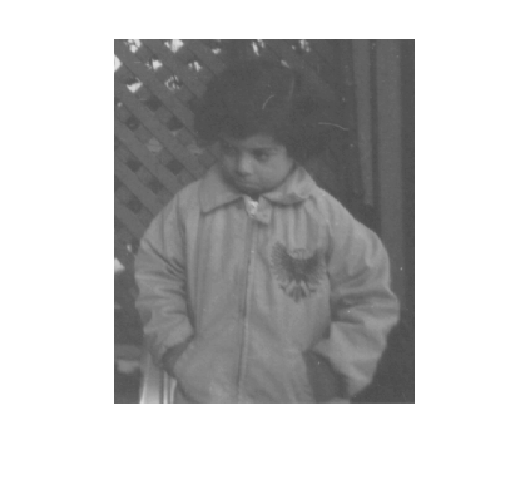}
         \caption{Original Pout Image}
     \end{subfigure}
     \begin{subfigure}[b]{0.55\textwidth}
         \includegraphics[width=\textwidth]{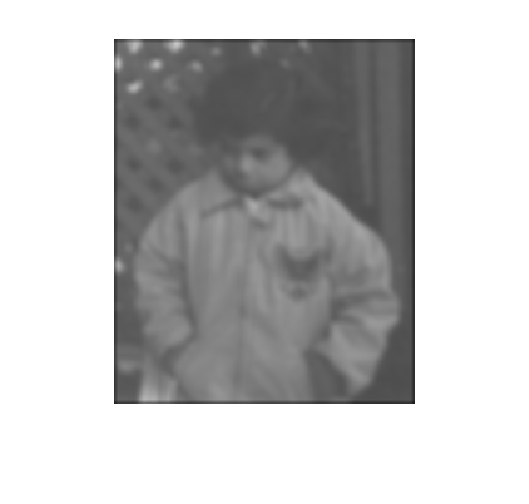}
         \caption{Degraded Pout Image}
     \end{subfigure}
     \begin{subfigure}[b]{0.55\textwidth}
         \includegraphics[width=\textwidth]{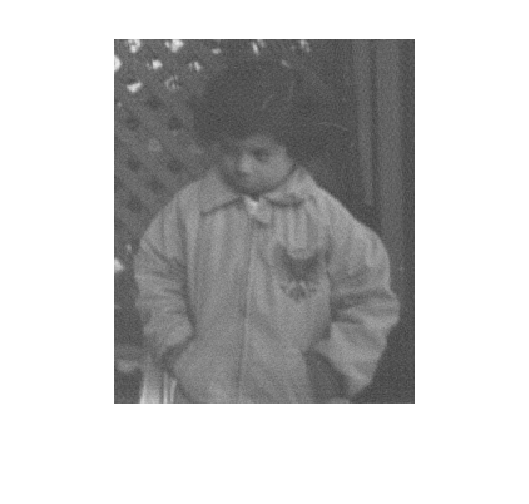}
         \caption{Proposed Alg. \ref{ALGG}}
     \end{subfigure}
     \begin{subfigure}[b]{0.55\textwidth}
         \includegraphics[width=\textwidth]{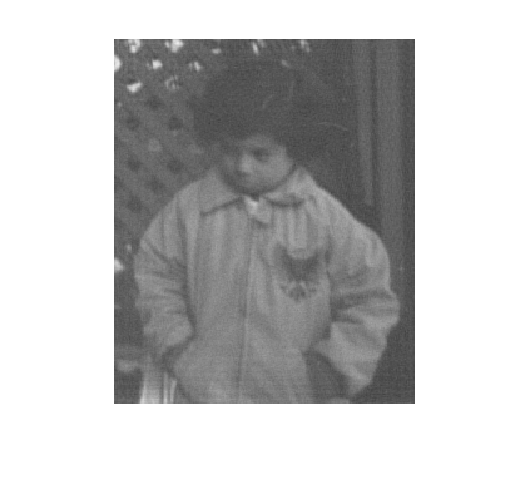}
         \caption{Damek \& Yin Alg.}
     \end{subfigure}
     \begin{subfigure}[b]{0.55\textwidth}
         \includegraphics[width=\textwidth]{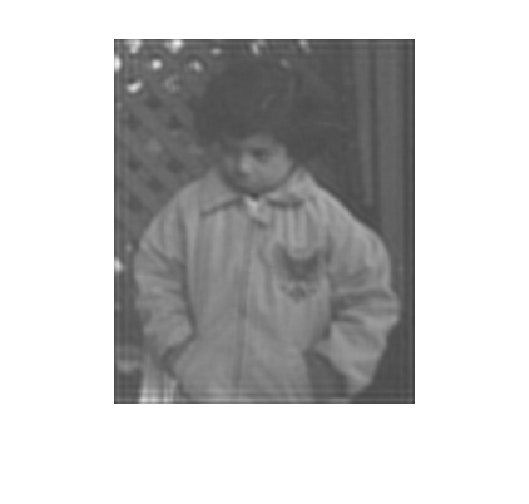}
         \caption{Artacho \& Belen Alg.}
     \end{subfigure}
     \begin{subfigure}[b]{0.55\textwidth}
         \includegraphics[width=\textwidth]{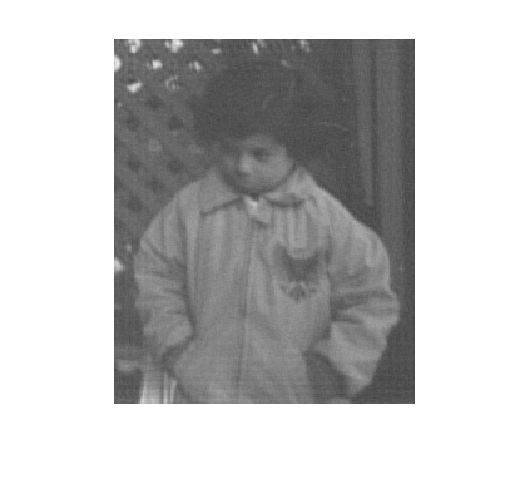}
         \caption{Beck \& Teboulle Alg.}
     \end{subfigure}
        \caption{Example \ref{Sec:LASSO} - Image recovery of Pout image by different methods.}
        \label{Pout_image}
\end{figure}

\begin{rem}$\,$\\
\begin{itemize}
\item Example \ref{Sec:LASSO} showcase the application of our proposed algorithm \ref{ALGG} in image recovery problems and the accuracy in recovering several images.
\item Our method outperformed many other methods proposed by Artacho \& Belen Alg. in \cite{Aragon-Artacho}, Beck \& Teboulle Alg. in \cite{Beck} and Damek \& Yin Alg. in \cite{DaY} after 1,000 iterations. Clearly in Table \ref{table:IRP} and Figures \ref{SNR_cam} - \ref{Pout_image}, our method produced higher signal-to-noise ratio over other methods, compared with, which indicate superiority of our method in recovering the images.
\end{itemize}
\end{rem}

\end{exm}

\begin{exm}\label{ex2}
\subsubsection*{Smoothly Clipped Absolute Deviation (SCAD) Penalty Problem}\label{Sec:SCAD}
Let us consider SCAD penalty problem \cite{SCAD}. Suppose $D\in\R^{M\times N},\ {\rm and}\ b\in\R^M$, the SCAD penalty problem is defined as
\begin{equation}\label{SCADPro}
\displaystyle\min_{u\in\R^N}\frac{1}{2}\|Du - b\|^2_2 + \displaystyle\sum_{k=1}^N q_\xi(|u_k|),
\end{equation}
where the penalty $q_\xi(\cdot)$ is defined as
\begin{equation*}
q_\xi(\omega):=\begin{cases}
\xi\omega,  \hspace{1.2cm} &\omega\leq\xi,\\
\frac{-\omega^2 + 2c\xi\omega - \xi^2}{2(c-1)}, \hspace{0.4cm} &\xi\leq\omega\leq c\xi,\\
\frac{(c+1)\xi^2}{2}, \hspace{0.4cm} &\omega\geq c\xi,
\end{cases}
\end{equation*}
where $c>2$ and the knots of the quadratic spline function is represented by $\xi > 0$.\\

\noindent Given that $\displaystyle \nu\geq\frac{1}{c-1}$, then the function $\displaystyle q_\xi(\cdot)+\frac{\nu}{2}|\cdot|^2$ is convex (see \cite[Theorem 3.1]{GuYZ}). As in \cite{WuLi}, we can formulate the SCAD penalty problem  in the form of problem \eqref{linprob}, with
\begin{align*}
g(u):=& \displaystyle\sum_{k=1}^N q_\xi(|u_k|) + \frac{1}{2(c-1)}\|u\|^2_2 \quad {\rm and}\\
f(u):=& \frac{1}{2}\|Du-b\|^2_2 - \frac{1}{2(c-1)}\|u\|^2_2.
\end{align*}
Accordingly, $f$ is convex and the gradient $\nabla g$ of $g$ is Lipschitz continuous  with Lipschitz constant 
$$L:=\max\left\{\|D^*D\|,\frac{1}{2(c-1)}\right\}.$$
\noindent By the Ballon-Haddad Theorem \cite[Corollary 18.17]{Bauschkebook} $\nabla g$ is $\left(\frac{1}{\|D^*D\|}\right)$-cocoercive. Now, \eqref{load1} can be formulated to this model as the following iterative scheme:

\begin{eqnarray}\label{eqaSCAD}
\left\{  \begin{array}{ll}
      & w_n = x_n + \theta(x_n - x_{n-1}) + \delta (x_{n-1} - x_{n-2});\\ \\
      & v_n = y_n - \gamma\left[D^*(Dy_n -b) - \frac{1}{(c-1)}y_n\right]\\ \\
      & z_n = {\rm arg}\displaystyle\min_{u\in\R^N}\left\{\displaystyle\sum_{k=1}^N q_\xi(|u_k|) + \frac{1}{2(c-1)}\|u\|^2_2 + \frac{1}{2\gamma}\|u-v_n\|^2_2\right\} \\ \\
      & x_{n+1}=(1-\rho)w_n+\rho z_n.
      \end{array}
      \right.
\end{eqnarray}
\noindent To conduct this numerical comparison, we take $x_{-1}=x_0=x_1\in\R^N$ with all entries as 1, we again randomly generate the data $b$. As detailed in \cite{SCAD}, the parameters $\tau$ and $c$ could be empirically gotten via generalized cross-validation or cross-validation techniques, by making use of sample data. In any case, in order to perform our numerical experiment, we set $c = 3.7$, $\tau = 0.1$ and use $||E_n||_2:=||x_n - x_{n+1}||\leq \varepsilon$ with $\varepsilon = 10^{-4}$. The rest of the parameters are stated in Table \ref{tab3}.

\begin{table}[H]
\caption{Methods Parameters for Example \ref{ex2}}
\centering
\begin{tabular}{c c c c c c c c c c c c c c}
 \toprule[1.5pt] \\
Proposed Alg. \ref{ALGG} & $\displaystyle \theta = 0.49$ & $\displaystyle \delta = -0.01$  & $\displaystyle \eta = \frac{1}{||D^*D||}$ & $\displaystyle \gamma = \eta$\\
\\
 & $\displaystyle \beta = \frac{2\gamma}{4\eta - \gamma}$ & $\displaystyle \rho = \frac{0.24}{\beta}$ & $\tau = 0.1$\\
\toprule[1.5pt]\\
Moudafi \& Oliny Alg. & $\displaystyle \theta_n = \frac{2}{||D^*D||}$ & $\displaystyle \gamma_n = 0.24\theta_n$ & $\displaystyle \alpha_n = \frac{n}{1000(n + 1)}$\\
\toprule[1.5pt] \\
Damek \& Yin Alg. & $\displaystyle \eta = \frac{1}{||D^*D||}$ & $\displaystyle \gamma = \eta$ & $\displaystyle \beta = \frac{2\gamma}{4\eta - \gamma}$ & $\displaystyle \rho = \frac{0.24}{\beta}$\\
\toprule[1.5pt] \\
Artacho \& Belen Alg. & $\displaystyle \eta = \frac{1}{||D^*D||}$ & $\displaystyle \gamma = 3.9\eta$ & $\displaystyle \beta = 2 - \frac{\gamma}{2\eta}$ & $\displaystyle \lambda = 0.24\beta$\\
\toprule[1.5pt] \\
Enyi et al. Alg. & $\displaystyle \gamma = \frac{1}{||D^*D||}$ & $\displaystyle \lambda = 0.3$ \\
\toprule[1.5pt] \\
Cevher \& Vu Alg. & $\displaystyle \gamma = \frac{\sqrt{2} - 1}{2||D^*D||}$ \\
\toprule[1.5pt]
\end{tabular}
\label{tab3}
\end{table}

\begin{table}[H]
\caption{Comparison of our proposed Algorithm \ref{ALGG} with other algorithms for the SCAD problem in different dimensions.}
\centering 
\begin{tabular}{c c c c c c c c c c c} 
\toprule[1.5pt]
 & \multicolumn{2}{c}{Case I: $200\times 1000$} & &  \multicolumn{2}{c}{Case II: $300\times 1200$} \\
  \cline{2-3}   \cline{5-6}
 & No. of Iter.  & CPU Time & & No. of Iter.  & CPU Time \\
 \toprule[1.5pt] \\
Proposed Alg. \ref{ALGG} & 58 & $1.8617\times 10^{-3}$ && 72 & $3.1522\times 10^{-3}$ \\ [0.5ex]
 \hline \\
Moudafi \& Oliny Alg. & 92 & $2.7931\times 10^{-3}$ && 116 & $5.1734\times 10^{-3}$ \\ [0.5ex]
 \hline\\
Damek \& Yin Alg. & 119 & $3.5443\times 10^{-3}$ && 148 & $6.9222\times 10^{-3}$ \\ [0.5ex]
\hline\\
Artacho \& Belen Alg. & 763 & $2.4867\times 10^{-2}$ && 904 & $3.9158\times 10^{-2}$ \\ [0.5ex]
 \hline\\
Enyi et al. Alg. & 101 & $3.1676\times 10^{-3}$ && 132 & $5.4510\times 10^{-3}$ \\ [0.5ex]
\hline\\
Cevher \& Vu Alg. & 196 & $5.8488\times 10^{-3}$ && 277 & $1.2362\times 10^{-2}$ \\ [0.5ex]
\toprule[1.5pt]
 & \multicolumn{2}{c}{Case III: $400\times 1400$} & &  \multicolumn{2}{c}{Case IV: $500\times 1600$} \\
  \cline{2-3}   \cline{5-6}
 & No. of Iter.  & CPU Time & & No. of Iter.  & CPU Time \\
 \toprule[1.5pt] \\
Proposed Alg. \ref{ALGG} & 87 & $5.9223\times 10^{-3}$ && 111 & $7.6624\times 10^{-3}$ \\ [0.5ex]
 \hline \\
Moudafi \& Oliny Alg. & 127 & $8.1677\times 10^{-3}$ && 147 & $1.1080\times 10^{-2}$ \\ [0.5ex]
 \hline\\
Damek \& Yin Alg. & 174 & $9.4503\times 10^{-3}$ && 202 & $1.4682\times 10^{-2}$ \\ [0.5ex]
\hline\\
Artacho \& Belen Alg. & 1060 & $6.1568\times 10^{-2}$ && 1255 & $9.4815\times 10^{-2}$ \\ [0.5ex]
 \hline\\
Enyi et al. Alg. & 164 & $1.0961\times 10^{-2}$ && 187 & $1.4637\times 10^{-2}$ \\ [0.5ex]
\hline\\
Cevher \& Vu Alg. & 314 & $1.7188\times 10^{-2}$ && 334 & $2.4966\times 10^{-2}$ \\ [0.5ex]
 \hline
\end{tabular}
\label{table:SCAD} 
\end{table}

\begin{figure}[H]
\minipage{0.50\textwidth}
 \includegraphics[width=\linewidth]{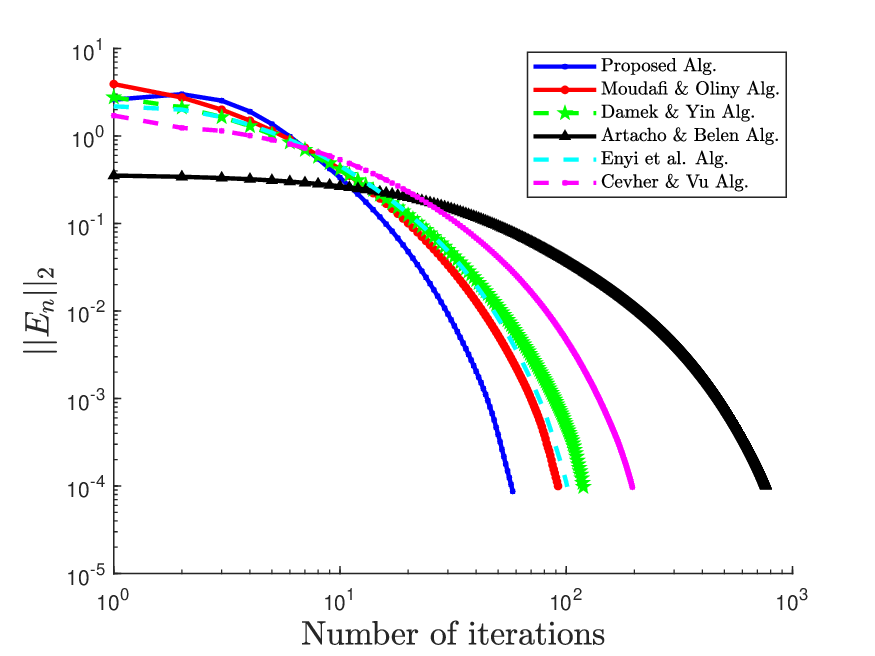}
 \subcaption{Case I}
\endminipage\hfill
\minipage{0.50\textwidth}
 \includegraphics[width=\linewidth]{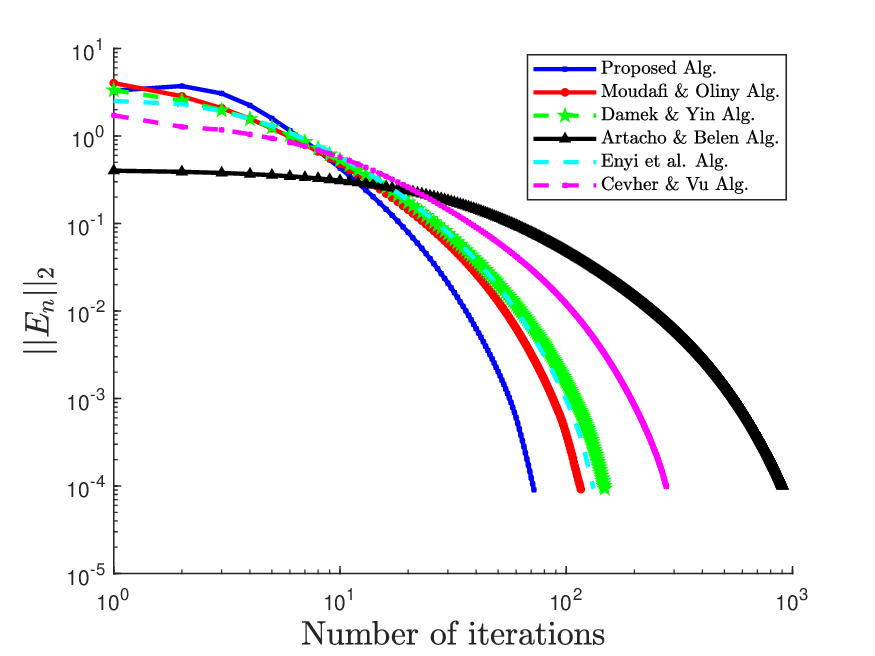}
  \subcaption{Case II}
\endminipage
\caption{SCAD problem: Comparison of different methods for Cases I \& II}\label{SCAD_1}
\end{figure}

\begin{figure}[H]
\minipage{0.45\textwidth}
 \includegraphics[width=\linewidth]{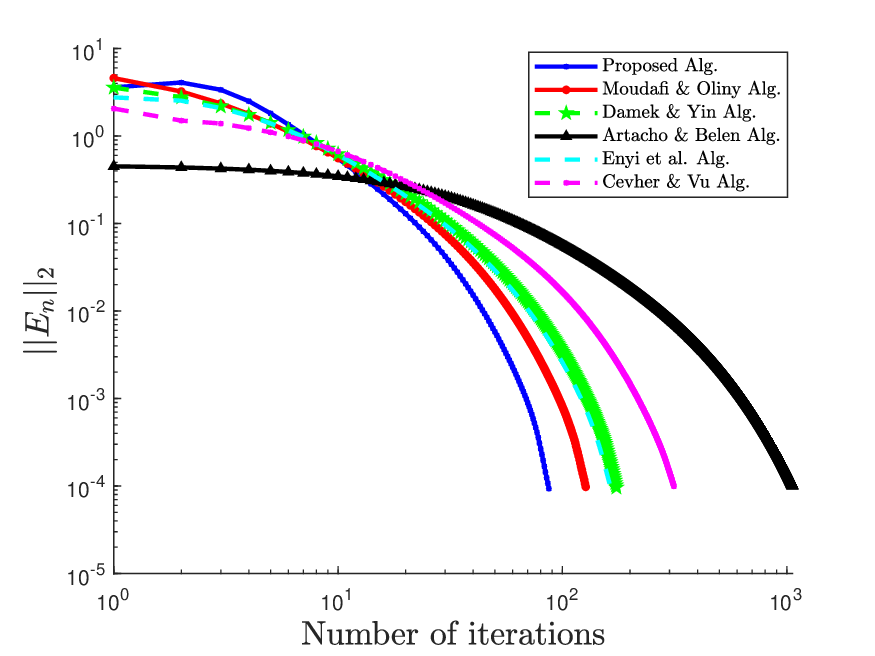}
 \subcaption{Case III}
\endminipage\hfill
\minipage{0.45\textwidth}
 \includegraphics[width=\linewidth]{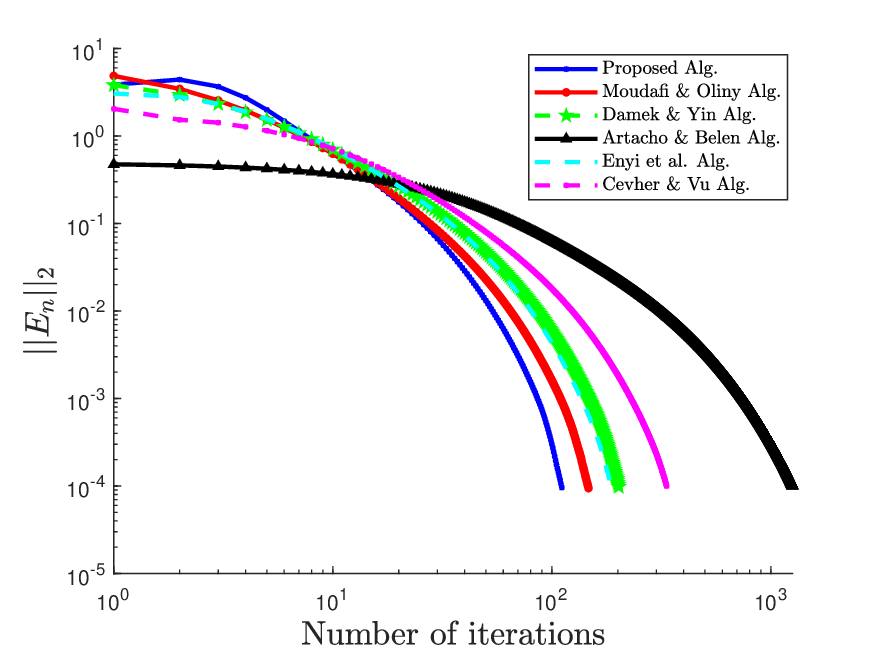}
  \subcaption{Case IV}
\endminipage
\caption{SCAD problem: Comparison of different methods for Cases III \& IV}\label{SCAD_2}
\end{figure}

\begin{figure}[H]
\minipage{0.50\textwidth}
 \includegraphics[width=\linewidth]{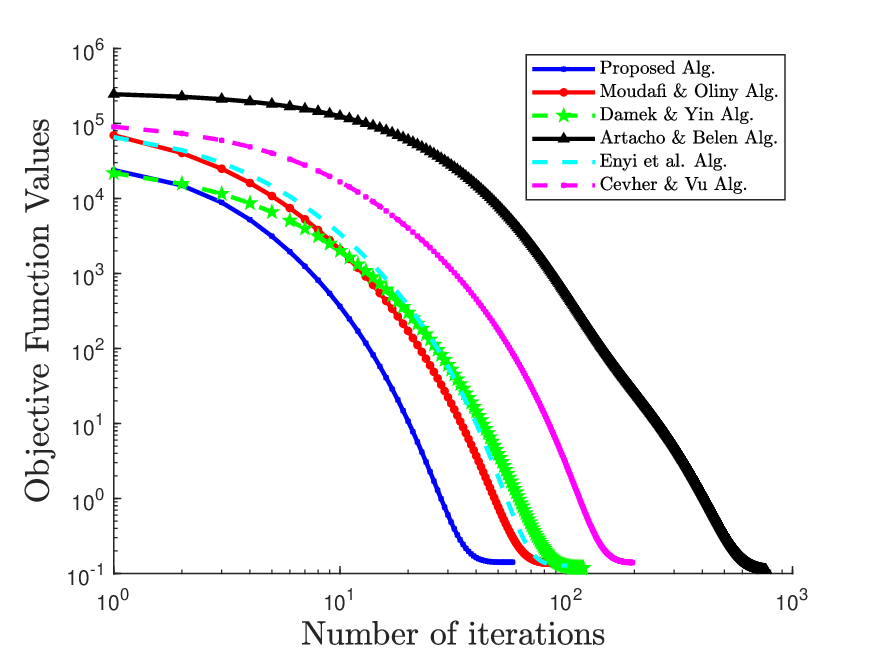}
 \subcaption{Case I}
\endminipage\hfill
\minipage{0.50\textwidth}
 \includegraphics[width=\linewidth]{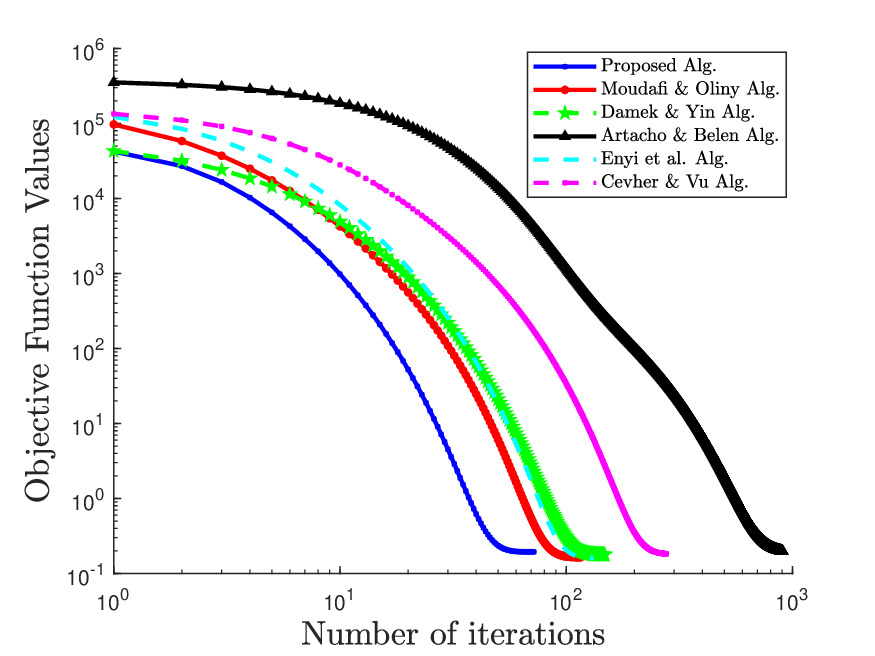}
  \subcaption{Case II}
\endminipage
\caption{SCAD problem: Comparison of different methods for Cases I \& II (Objective Function)}\label{SCAD_3}
\end{figure}

\begin{figure}[H]
\minipage{0.45\textwidth}
 \includegraphics[width=\linewidth]{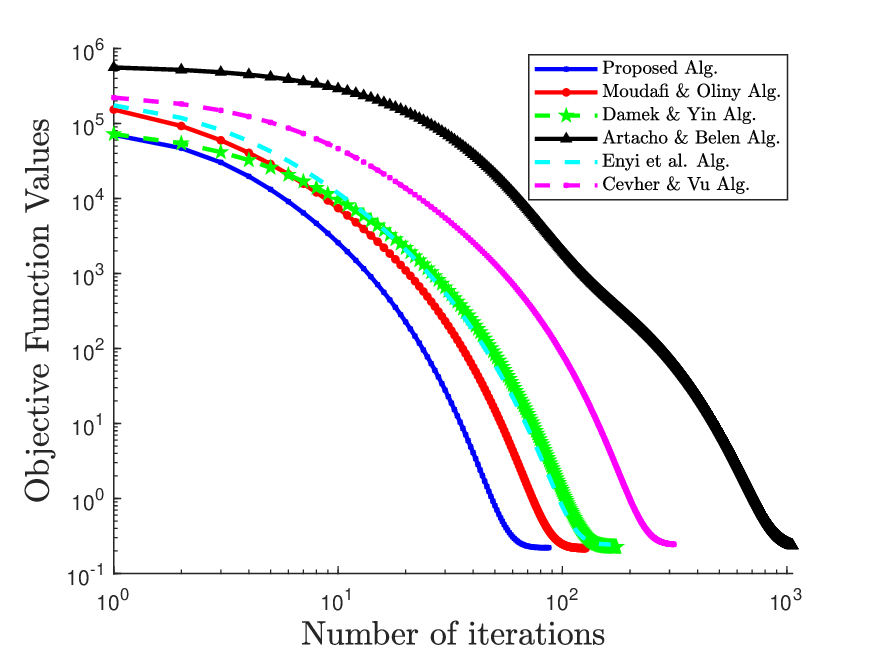}
 \subcaption{Case III}
\endminipage\hfill
\minipage{0.45\textwidth}
 \includegraphics[width=\linewidth]{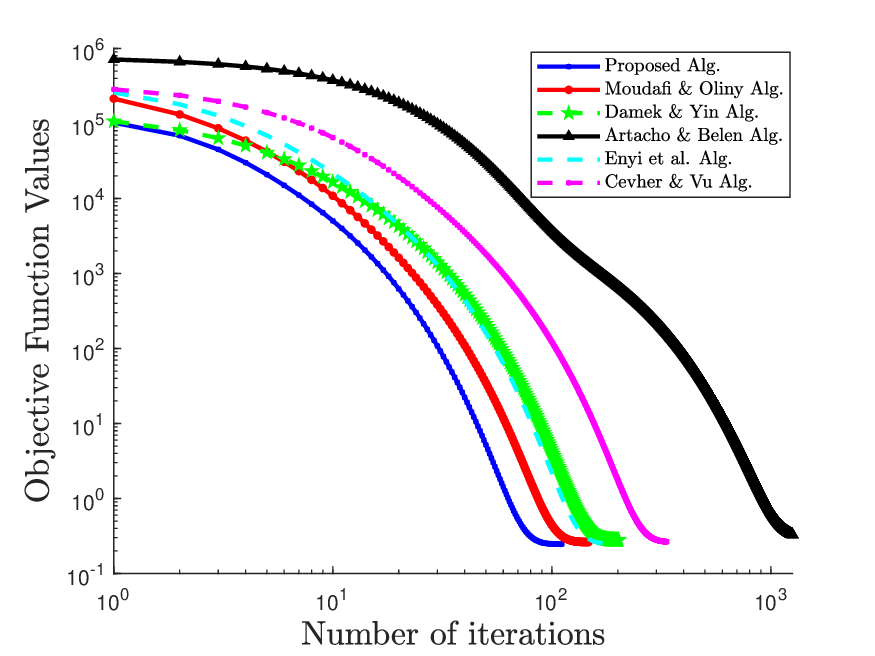}
  \subcaption{Case IV}
\endminipage
\caption{SCAD problem: Comparison of different methods for Cases III \& IV (Objective Function)}\label{SCAD_4}
\end{figure}

\begin{rem}$\,$\\
\begin{itemize}
\item We demonstrate the effectiveness and easy-to-implement nature of our proposed algorithm through Example \ref{ex2} for Smoothly Clipped Absolute Deviation (SCAD) Penalty Problem.
\item Clearly, our method outperformed several state of the heart methods proposed by Moudafi \& Oliny Alg. in \cite{ModOli}, Damek \& Yin Alg. in \cite{DaY}, Artacho \& Belen Alg. in \cite{Aragon-Artacho}, Enyi et al. Alg. in \cite{enyi} and Cevher \& Vu Alg. in \cite{CevVu} in both CPU time and the number of iterations. (see Table \ref{table:SCAD} and Figures \ref{SCAD_1} - \ref{SCAD_4}).
\end{itemize}
\end{rem}
\end{exm}

\section{ Final Remarks}\label{Sec6}
This paper proposes a method which is a combination of two-step inertial extrapolation steps, relaxation parameter and splitting algorithm of Davis and Yin to find a zero of the sum of three monotone operators for which two are maximal monotone and the third co-coercive. Weak convergence results are obtained without assuming summability conditions imposed on inertial parameters and the sequence of iterates assumed in recent results on multi-step inertial methods in the literature, and some previously known related methods in the literature are recovered. Numerical illustrations with test problems drawn from Image Restoration Problem (IRP) and Smoothly Clipped Absolute Deviation (SCAD) Penalty Problem in statistical learning are given to illustrate the presented convergence theory. These examples show that our proposed algorithm is easy to implement and effective in handling important application like IRP and SCAD problems. One of the contributions of this paper is to improve on the setbacks observed recently in the literature that one-step inertial Douglas-Rachford splitting method may fail to provide acceleration.\\

\noindent As part of our future project, we intend to further consider our proposed method with correction terms to solve the same problem considered in this paper.

\end{document}